\documentclass[11pt]{amsart}
\usepackage[nobysame,abbrev,alphabetic]{amsrefs}
\usepackage{amssymb}
\usepackage{amsmath}
\usepackage[all]{xy}
\usepackage{verbatim}
\usepackage{float}
\restylefloat{figure}
\usepackage{tikz}

\usepackage{kantlipsum} 

\calclayout

\usepackage[colorlinks=true, pdfstartview=FitV, linkcolor=blue,
  citecolor=blue, urlcolor=blue,pagebackref=false]{hyperref}

\DeclareMathOperator{\Img}{Im}

\newtheorem{thm}{Theorem}[section]
\newtheorem{claim}[thm]{Claim}

\newtheorem{lem}[thm]{Lemma}
\newtheorem{prop}[thm]{Proposition}
\newtheorem{defin}[thm]{Definition}

\newtheorem{examples}[thm]{Examples}
\newtheorem{rem}[thm]{Remark}
\swapnumbers

\numberwithin{equation}{section}

\newcommand\Lref[1]{{Lemma~\ref{#1}}}%
\newcommand\Clref[1]{{Claim~\ref{#1}}}%
\newcommand\Pref[1]{{Proposition~\ref{#1}}}%
\newcommand\Cref[1]{{Corollary~\ref{#1}}}%
\newcommand\Tref[1]{{Theorem~\ref{#1}}}%
\newcommand{\nati}[1]{{\color{blue}\bf #1}}
\newcommand{\michael}[1]{{\color{red}\bf #1}}

\title{Expander Graphs -- Both Local and Global}
\makeatletter
\author[Michael Chapman]{Michael Chapman}
\address{Michael Chapman\hfill\break Einstein institute of Mathematics,\hfill\break The Hebrew University, Jerusalem 91904, Israel.}
\email{michael.chapman@mail.huji.ac.il}

\author[Nati Linial]{Nati Linial }
\address{Nati Linial\hfill\break Department of Computer Science,\hfill\break The Hebrew University, Jerusalem 91904, Israel. }
\email{nati@cs.huji.ac.il }
\thanks{The second author is supported by ERC grant 339096 ``High dimensional combinatorics".}

\author[Yuval Peled]{Yuval Peled}
\address{Yuval Peled\hfill\break Courant Institute of Mathematical Sciences,\hfill\break
New York University\\ 251 Mercer St.\\ New York, NY~10012\\~USA.}
\email{yuval.peled@courant.nyu.edu}
\makeatother

\keywords{}

\begin{document}

\maketitle

\begin{abstract}
Let $G=(V,E)$ be a finite graph. For $v\in V$ we denote by $G_v$ the subgraph of $G$ that is induced by $v$'s neighbor set. We say that $G$ is $(a,b)$-regular for $a>b>0$ integers, if $G$ is $a$-regular and $G_v$ is $b$-regular for every $v\in V$. Recent advances in PCP theory call for the construction of infinitely many  $(a,b)$-regular expander graphs $G$ that are expanders also locally. Namely, all the graphs $\{G_v|v\in V\}$ should be expanders as well. While random regular graphs are expanders with high probability, they almost surely fail to expand locally. Here we construct two families of $(a,b)$-regular graphs that expand both locally and globally. We also analyze the possible local and global spectral gaps of $(a,b)$-regular graphs. In addition, we examine our constructions vis-a-vis properties which are considered characteristic of high-dimensional expanders.
\end{abstract}

\section{Introduction}
It is hard to overstate the significance of expander graphs in theoretical computer science and the impact their study has had on a number of mathematical areas. A particularly fascinating example of such an application is Dinur's proof of the PCP Theorem, e.g., \cite{RS07}. However, in recent advances in PCP theory \cite{DinurKaufman} more specialized expander graphs are required. If $v$ is a vertex in a graph $G$ we denote by $G_v$ the subgraph of $G$ that is induced by $v$'s neighbors and call it the \emph{link} of $v$ in $G$. We seek large regular expanders $G$ such that $G_v$ is an expander for every $v\in V(G)$. 

One of the first discoveries in the study of expanders is that for every $d\ge 3$ asymptotically almost every $d$-regular graph is a very good expander. However, it is easy to verify that almost every $d$-regular graph is very far from satisfying the above requirement, as $G_v$ is typically an anticlique. So, here is the central question of the present article: Given positive integers $a>b$ do there exist arbitrarily large $(a,b)$-expanders? Namely, $a$-regular expander graphs $G$ such that every $G_v$ is a $b$-regular expander. If so, how good can the expansion properties (edge expansion, spectral gap) of $G$ and the graphs $G_v$ be?

These investigations are closely related to the recently emerging field of high-dimensional expanders. Vertex-expansion, edge-expansion, spectral gaps and the speed of convergence of the simple random walk on the graph are key ingredients in the theory of expander graphs. While these parameters need not perfectly coincide, they mutually control each other quite tightly. In contrast, the high-dimensional theory suggests a number of inherently different ways to quantify expansion. Namely, the connections between these concepts are nowhere as tight as in the one-dimensional case of expander graphs. It is very suggestive to explore families of $(a,b)$-expanders in light of this array of quantitative measures of high-dimensional expansion. 

\subsection*{Preliminaries, main results and organization}
Let $G$ be a graph and $v\in V(G)$. The link of $v$ denoted $G_v$ is the subgraph of $G$ that is induced by the vertex set $\{u\in V\ |\ uv\in E\}$.

\begin{defin}
Let $a>b\geq 0$ be integers. An {\em $(a,b)$-regular graph} $G$ is an $a$-regular graph, where for every vertex $v\in V(G)$ the link $G_v$ is $b$-regular.
\end{defin}

We recall some basic notions about expander graphs. Let $G$ be a $d$-regular graph with adjacency matrix $A_G$, and let $d=\lambda_1\geq  \lambda_2 \geq ...\geq \lambda_n$ be $A_G$'s eigenvalues. We say that $G$ is an \emph{$\epsilon$-spectral expander} if its \emph{normalized spectral gap} is at least $\epsilon$, i.e., $1-\frac{\lambda_2}{d}\geq\epsilon$. We say that $G$ is a \emph{$\delta$-edge expander} if $|E(U,V\setminus U)|\ge\delta\cdot \min(|U|,|V\setminus U|)$ for every $U\subseteq V(G)$, where $E(A,B)$ is the set of edges with one vertex in $A$ and one in $B$. The largest such $\delta$ is called the edge-expansion (or Cheeger constant) of $G$. 
In our general discussion, we occasionally say that {\em $G$ is an expander} or that {\em it is an $(a,b)$-expander}. These expressions mean respectively that $G$ has some non-trivial but unspecified spectral gap and that $G$ is $(a,b)$-regular and has some non-trivial but unspecified local and global spectral gaps. In the more technical parts of the paper we avoid such loose language and specify the parameters as needed.

\begin{examples} \rm Here are some known families of $(a,b)$-regular graphs:
\begin{enumerate}
    \item An $n$-clique is $(n-1,n-2)$-regular and has good expansion properties. However, a family of $(n-1,n-2)$-regular graphs must consist of disjoint unions of $n$-cliques, and such graphs fail to be global expanders. 
    \item Let $p$ be a prime integer. The $2$-dimensional Ramanujan complexes coming from $PGL_3(\mathbb{Q}_p)$ are $(2p^2+2p+2,p+1)$-regular and have good local and global expansion properties (See \cite{LSV} and \cite{KostaOri}). Note that this is the \textbf{only} family of Ramanujan complexes whose $1$-skeleton is $(a,b)$-regular for some $a$ and $b$. These graphs have many high dimensional expansion properties, see e.g., \cite{DinurKaufman} and \cite{EvraKaufman}.  
      \item An additional group theoretic construction is due to Kaufman and Oppenheim in \cite{Kaufman_Opp1} (The $2$-dimensional case). 
    \item Though $a$ and $b$ are not bounded in this case, Conlon's hypergraph expanders \cite{Conlon} are also $(a,b)$-expanders.
    \item The $1$-skeleton of non-singular, $a$-regular triangulations of surfaces is $(a,2)$-regular. See Section \ref{sec:geom_const} for more on this.
    
\end{enumerate}
\end{examples}

In section \ref{secABTB} we ask how large the spectral gaps can be in an $(a,b)$-regular graph. We first prove an optimal Alon-Boppana type bound which makes no reference to the graph's local expansion:

\begin{thm} \label{ABTB}
The second eigenvalue of an $(a,b)$-regular graph satisfies $$\lambda_2\geq b+2\sqrt{a-b-1}-o_n(1).$$
The bound is tight.
\end{thm}

In the graphs that we construct to prove the tightness of the bound in \Tref{ABTB}, all the links are disconnected. Therefore, it is natural to ask whether the same bound can be attained by graphs whose links are all expanders, or at least connected. The following theorem shows that the answer is negative, by describing some {\em tradeoff} between local and global expansion.

\begin{thm} \label{tradeoff}
Consider an $(a,b)$-regular graph each of whose links has edge expansion at least $\delta>0$. Then {there exists some $\epsilon=\epsilon(a,b,\delta)>0$ such that the} second eigenvalue {of the graph} satisfies:
\[
\lambda_2 \geq \left( b+2\sqrt{a-b-1}\right)(1+\epsilon)-o_n(1).
\]
For fixed $a$ and $b$ with $a\geq b^2+O(b)$, $\epsilon$ strictly increases with $\delta$. For any other {fixed} values of $a$ and $b$, $\epsilon$ increases for small enough $\delta$.
\end{thm}

In Section \ref{sec:polyg} we introduce our \emph{Polygraph} constructions, which can be viewed as a family of new graph products. Concretely, these constructions transform a high-girth regular expander into an $(a,b)$-expander.
To illustrate this idea, let $q> p \ge 0$ be integers, let $G$ be a graph {with distance function $\rho$} and girth larger then $3p+3q$. The vertex set of the {\em polygraph} $G_S$ is $V(G)^3$ and $(x_1, x_2, x_3)$ is a neighbor of $(y_1, y_2, y_3)$ iff the multiset of three distances $[\rho(x_i,y_i)|i=1,2,3]$ coincides with the multiset $[p, q, p+q]$.

For illustration, here is a way of viewing the Polygraph corresponding to $p=0$ and $q=1$. Take three copies of a $d$-regular graph $G$ of girth bigger than $3$ and have a token move on each of them. At every step two of the tokens move to a neighboring vertex and the third token stays put. Any configuration of tokens is a vertex of the graph and the above process defines its adjacency relation.

\begin{thm} \label{Poly_are_exp}
Let $q> p\ge 0$ be even integers. 
If $G$ is connected, non-bipartite and its girth is bigger than $3p+3q$, then $G_S$ is an $(a,b)$-regular local $\epsilon$-spectral expander and global $\epsilon'$-spectral expander. Here $a,b$ and $\epsilon$ depend only on $p$ and $q$, while $\epsilon'$ depends also on the spectral gap of $G$.
\end{thm}

In Section \ref{sec:concrete} we investigate in detail the regularity and local spectral gaps of two specific Polygraph constructions.
In Section \ref{sec:hdx} we examine Polygraphs from the perspective of high-dimensional expanders. We discuss properties such as geometric overlap, discrepancy, coboundary expansion and mixing of the edge-triangle-edge random walks.
In Section \ref{sec:geom_const} we provide some additional constructions of $(a,b)$-regular graphs, based on regular triangulations of surfaces and tensor products of graphs.
We conclude the paper with some open questions related to this study.

\section{The second eigenvalue of $(a,b)$-regular graphs - Lower bounds} \label{secABTB}

\begin{proof}[Proof of \Tref{ABTB}]
This theorem is reminiscent of the Alon-Boppana Theorem. We are inspired by the proof of that theorem via the moment method (e.g., \cite{HLV} Section 5.2). Let $G$ be a $d$-regular graph with adjacency matrix $A_G=A$ and eigenvalues $d=\lambda_1 \ge \lambda_2 \ge \ldots \ge \lambda_n$. For $t$ a positive integer we note that
\begin{equation}\label{moment_method}
\textrm{trace}(A^{t})=\sum \lambda_i^t \le d^t+(n-1)\cdot \Lambda^t,
\end{equation}
where $\Lambda=\lambda(G):=\max\{\lambda_2, -\lambda_n\}$. On the other hand, $\textrm{trace}(A^{t})$ is the number of closed walks of length $t$ in $G$. This number can be bounded from below by counting length-$t$ walks that start and end at some given {\em origin} vertex in $G$'s universal cover $\mathbb{T}_d$, the (infinite) $d$-regular tree. Associated with such a walk is a word in $\{F, B\}^t$, where $F$ (resp.\ $B$) stands for a {\em forward} step away from the origin ({\em backward} step toward it). This word satisfies the {\em Catalan condition}, i.e., it has an equal number of $B$'s and $F$'s, and every initial segment has at least as many $F$'s as $B$'s. Also, $B$-steps are uniquely defined whereas every $F$-step can be realized in $d-1$ ways. By working out the number of such words, the Alon-Boppana bound $\Lambda\ge 2\sqrt{d-1}-o_n(1)$ is obtained.

En route to a proof there are two obvious obstacles:
\begin{equation}\label{origin}
\text{When the walk resides at the origin, there are $d$ possible $F$ steps.}
\end{equation}
We soon address this point.
\begin{equation}\label{parity}
\noindent\text{A closed walk in $\mathbb{T}_d$ has even length, but in our proof $t$ is odd.}
\end{equation}
The advantage of $t$ being  odd is that the term $\Lambda^t$ in (\ref{moment_method}) can be replaced by the possibly smaller $\lambda_2^t$.

Our proof needs a modified notion of forward and backward steps and also allow for {\em sideways} step. We consider length-$t$ {\em Catalan words} in the alphabet
$\Sigma=\{F_j|j=1,\ldots,a-b-1\}\cup\{B\}\cup\{S_i|i=1,\ldots,b\}.$
Namely, a word with an equal number of $F$'s and $B$'s where $\#F\ge\#B$ in each initial segment. We wish to injectively associate to each such word a closed walk in our graph. Roughly speaking, when the next letter in the word is $F_j$ we should move to the $j$-th {\em forward neighbor} of our current position, likewise move to the $i$-th {\em sideways neighbor} upon reading $S_i$, and finally moving one step backward on a $B$.

In $\mathbb{T}_d$ it is perfectly clear what forward and backward mean and sideways does not exist. As we explain next, we navigate a general graph using a {\em local system of coordinates}. To this end we use a stack $X$ in which we store vertices, where every two consecutive entries in $X$ are two adjacent vertices in $G$. An invariant that we maintain is that $x$, the vertex at which the walk currently resides is always a neighbor of {\bf top}, the top entry of $X$. Suppose that we move next from $x$ to a neighbor $y$. 
\begin{itemize}
\item
If $y$ is not a neighbor of {\bf top}, this is a {\em forward} step, and we {\em push} $x$.
\item
If $y=${\bf ~top}, this is a {\em backward} step and we {\em pop}.
\item
If $y$ is a neighbor of {\bf top}, this is a {\em sideways} step, and the stack stays unchanged.
\end{itemize}

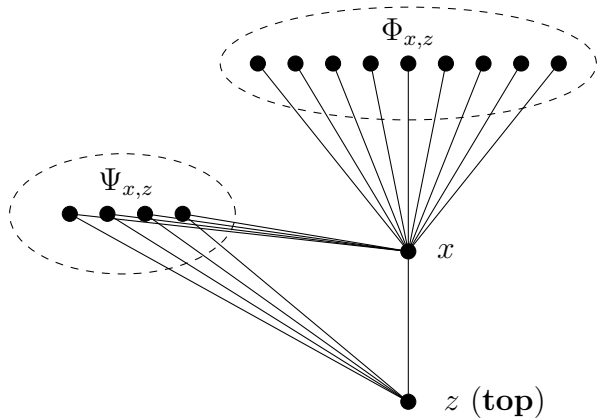
\begin{figure}
\begin{center}
\begin{tikzpicture}
\draw[dashed] (4,2.5) ellipse (2.5cm and .75cm);
\draw[dashed] (0.2,.5) ellipse (1.5cm and 0.8cm);
\draw[fill=black] (4,0) circle (0.1cm);
\draw[fill=black] (4,-2) circle (0.1cm);
\draw[fill=black] (1,0.5) circle (0.1cm);
\draw[fill=black] (0.5,0.5) circle (0.1cm);
\draw[fill=black] (0,0.5) circle (0.1cm);
\draw[fill=black] (-0.5,0.5) circle (0.1cm);

\draw[fill=black] (4,2.5) circle (0.1cm);
\draw[fill=black] (4.5,2.5) circle (0.1cm);
\draw[fill=black] (3.5,2.5) circle (0.1cm);
\draw[fill=black] (3,2.5) circle (0.1cm);
\draw[fill=black] (5.5,2.5) circle (0.1cm);
\draw[fill=black] (5,2.5) circle (0.1cm);
\draw[fill=black] (2.5,2.5) circle (0.1cm);
\draw[fill=black] (6,2.5) circle (0.1cm);
\draw[fill=black] (2,2.5) circle (0.1cm);

\draw (4,0)--(4,2.5);
\draw (4,0)--(4.5,2.5);
\draw (4,0)--(3.5,2.5);
\draw (4,0)--(3,2.5);
\draw (4,0)--(2.5,2.5);
\draw (4,0)--(2,2.5);
\draw (4,0)--(5,2.5);
\draw (4,0)--(5.5,2.5);
\draw (4,0)--(6,2.5);

\draw (4,-2)--(4,0);
\draw (4,-2)--(1,0.5)--(4,0);
\draw (4,-2)--(0.5,0.5)--(4,0);
\draw (4,-2)--(0,0.5)--(4,0);
\draw (4,-2)--(-0.5,0.5)--(4,0);

\node at (4.5,0) {\large $x$};
\node at (5.15,-2) {\large $z~$({\bf top})};

\node at (4,2.9) {\large $\Phi_{x,z}$};
\node at (0.25,0.9) {\large $\Psi_{x,z}$};

\end{tikzpicture}
\end{center}
\caption{
The local system of coordinates of $x$, at which the walk currently resides, with respect to its neighbor $z$, that is currently at the top of the stack.}
\label{fig:local}
\end{figure}

It remains to define which is $x$'s $j$-th forward (resp.\ sideways) neighbor. This choice is not absolute, but rather depends on the current {\bf top}: Consider two neighbors $x$ and $z$ in $G$, where we think of $z$ as the current {\bf top}, and $x$ as our current position. We fix some ordering on the set $\Phi_{x,z}$ of the $a-b-1$ neighbors of $x$ that are not neighbors of $z$ and an ordering on the set $\Psi_{x,z}$ of the $b$ joint neighbors of $x$ and $z$. Thus, if we are currently at $x$, and $z$ is at the top of the stack, we interpret the symbol $F_j$ as ``move to the $j$-th vertex in $\Phi_{x,z}$". Likewise $S_i$ means ``move from $x$ to the $i$-th vertex in $\Psi_{x,z}$" and $B$ means ``step from $x$ to $z$". In other words, $\Psi_{x,z}= G_x\cap G_z$, whereas $\Phi_{x,z}= G_x\setminus (G_z\cup\{z\})$. (See Figure \ref{fig:local}).

Given a starting vertex $v$ and a length-$t$ Catalan-word $\omega$ over the alphabet $\Sigma$, we will specify a closed walk of length $t+2$ that starts and ends at $v$. Before we do that we need to deal with the issue raised in (\ref{origin}) above.
We associate with every vertex $v$ one of its neighbors $\varphi_v$. We start our walk at $v$, then move to $\varphi_v$ and push $v$ on the stack. Henceforth we follow the transitions that are dictated by $\omega$ and the push/pop rules described above. Since $\#F_{\omega}=\#B_{\omega}$, when we are done reading $\omega$ the stack contains only the symbol $v$, and we therefore reside at a neighbor of $v$. We now empty the stack and move to $v$. This clearly associates injectively a closed path as described with every pair $(v,\omega)$ for $v$ a vertex and $\omega$ a Catalan word.

We have thus shown that the number of closed walks of length $t+2$ in $G$ is at least
\begin{equation}
n\sum_{0\leq k< \frac{t}{2}}\binom{t}{k, k, t-2k} \frac{1}{k+1} b^{t-2k}(a-b-1)^k.
\end{equation}

Here $n$ counts the choices of the starting vertex $v$. The trinomial coefficient counts words $\omega$ with $\#F_{\omega}=\#B_{\omega}=k$ and $\#S_{\omega}=t-2k$. The term $\frac{1}{k+1}$ accounts for the probability that the Catalan Condition holds. Finally every $F$-step can be indexed in $a-b-1$ ways and every $S$-step in $b$ ways.

There are only $O(t)$ terms in this sum whereas the largest term is exponential in $t$, so it suffices to determine the largest term in the sum. To this end we express $k=\alpha t$, and then we need to find the $\alpha$ that maximizes the expression\footnote{Logarithms here are to base $2$, unless otherwise stated.}
\begin{equation}\label{disco}
\mathcal{H}(\alpha,\alpha,1-2\alpha)+\alpha \log(a-b-1) + (1-2\alpha)\log b,
\end{equation}
where $\mathcal{H}$ is the binary entropy function.
Straightforward calculation yields that the maximum is $\log(b+2\sqrt{a-b-1})$ which is attained for $\alpha=\frac{\sqrt{a-b-1}}{b+2\sqrt{a-b-1}}$.
When we return to (\ref{moment_method}), the best lower bound on $\lambda_2$ is attained for $t\approx \frac{\log n}{\log d - \log(\lambda_2)}$ and yields
\[
\lambda_2\ge b+2\sqrt{a-b-1}-o_t(1) 
\]
as claimed.

We now prove that this bound is tight. Let $H=(L,R,F)$ be a connected, bipartite, left $c$-regular and right $d$-regular graph with $\textrm{girth}(H)\ge 8$. Associated with $H$ is the $c(d-1)$-regular graph $G=(L,{E})$, where $xy\in E$ if and only if there is a vertex $z\in R$ such that $xz, yz\in F$. Note that every link in $G$ is a $(d-2)$-regular graph which is the disjoint union of $c$ graphs each of which a $(d-1)$-clique. All told, this is a construction of $(c(d-1),d-2)$-regular graphs. For example, here is a concrete family of bipartite graphs $H$ as above with $c=2$. Let $\Gamma$ be a $d$-regular triangle-free graph, and let $L=E(\Gamma), R= V(\Gamma)$ and $F$ the vertex-edge incidence relation of $\Gamma$.

Of course, the links in this graph are not expanders - they are not even connected. It is easy to see that the adjacency matrix of $G$ is a block in $A_H^2-cI$. If $H$ is a $(c,d)$-biregular bipartite Ramanujan graph  (see Section $2.3$ in \cite{MarSpiSri}), then $\lambda(G)\leq \sqrt{c-1}+\sqrt{d-1}$. Thus $$\lambda(G)\le (\sqrt{c-1}+\sqrt{d-1})^2 -c=d-2+2\sqrt{c(d-1)-(d-2)-1}$$ showing that the bound is tight.
\end{proof}

\begin{proof}[Proof of \Tref{tradeoff}]

To start, we improve the lower bound on $\lambda_2$ in Theorem \ref{ABTB} when each $G_v$ is {\em connected}. In those cases where Theorem \ref{ABTB} gives a tight bound, our census of closed walks is complete. However, as we soon observe, when all the links are connected, many additional closed walks emerge. To maintain the overall structure of the proof $F, B$ and $S$ steps still go with push, pop and no change to the stack, but they need no longer reflect the distance from the origin.
 
Given an initial vertex $v$ and a word in $B, F, S$ (with appropriate indices) again we associate to these data a walk in $G$ that starts and ends in $v$. However, the correspondence is now somewhat different. Suppose that the walk currently resides at the vertex $x$, its neighbor $z$ is at the top of the stack, and the coming two letters are $FB$ in this order. Because the link of $x$ is connected, there must be an edge between some vertex $u\in\Phi_{x,z}$ and some vertex $y\in \Psi_{x,z}$. Say that we realize the $F$-step by moving from $x$ to $u$. After this move {\bf top~}$=x$ and the penultimate entry in the stack is $z$. In the proof of Theorem \ref{ABTB} the coming $B$-step is realized now by moving back to $x$, and popping $x$, making {\bf top~}$=z$. But because $u$ has a neighbor $y\in \Psi_{x,z}$, we can also move from $u$ to $y$ and pop $x$ while respecting the structure of the proof. In other words, now we can and will consider the transitions $x\to u\to y$ as realizing the subword $FB$ rather than $FS$ (see Figure \ref{fig:local_exp}).

\begin{figure}
\begin{center}
\begin{tikzpicture}
\draw[dashed] (4,2.5) ellipse (2.5cm and .85cm);
\draw[dashed] (0.2,.5) ellipse (1.5cm and 0.8cm);
\draw[fill=black] (4,0) circle (0.1cm);
\draw[fill=black] (4,-2) circle (0.1cm);
\draw[fill=black] (1,0.5) circle (0.1cm);
\draw[fill=black] (0.5,0.5) circle (0.1cm);
\draw[fill=black] (0,0.5) circle (0.1cm);
\draw[fill=black] (-0.5,0.5) circle (0.1cm);

\draw[fill=black] (4,2.5) circle (0.1cm);
\draw[fill=black] (4.5,2.5) circle (0.1cm);
\draw[fill=black] (3.5,2.5) circle (0.1cm);
\draw[fill=black] (3,2.5) circle (0.1cm);
\draw[fill=black] (5.5,2.5) circle (0.1cm);
\draw[fill=black] (5,2.5) circle (0.1cm);
\draw[fill=black] (2.5,2.5) circle (0.1cm);
\draw[fill=black] (6,2.5) circle (0.1cm);
\draw[fill=black] (2,2.5) circle (0.1cm);

\draw (4,0)--(4,2.5);
\draw (4,0)--(4.5,2.5);
\draw (4,0)--(3.5,2.5);
\draw (4,0)--(3,2.5);
\draw (4,0)--(2.5,2.5);
\draw (4,0)--(2,2.5);
\draw (4,0)--(5,2.5);
\draw (4,0)--(5.5,2.5);
\draw (4,0)--(6,2.5);

\draw (4,-2)--(4,0);
\draw (4,-2)--(1,0.5)--(4,0);
\draw (4,-2)--(0.5,0.5)--(4,0);
\draw (4,-2)--(0,0.5)--(4,0);
\draw (4,-2)--(-0.5,0.5)--(4,0);

\node at (4.5,0) {\large $x$};
\node at (5.15,-2) {\large $z~$({\bf top})};
\node at (2,2.8) {\large $u$};
\node at (-.5,0.8) {\large $y$};

\draw[red,ultra thick,dashed] (-0.5,.5)--(2,2.5);
\node at (4,2.9) {\large $\Phi_{x,z}$};
\node at (0.75,0.9) {\large $\Psi_{x,z}$};

\end{tikzpicture}
\end{center}
\caption{
The assumption that $G_x$ is connected guarantees that there is an edge $uy$ between $\Phi_{x,z}$ and $\Psi_{x,z}$. The walk $x\to u \to y$ is now considered an $FB$ move rather than an $FS$ move as in \Tref{ABTB}. As a result, the vertex $z$, and not $x$, is on the top of the stack when the walk reaches $y$. The better an expander graph $G_x$ is, the more edges there are between $\Phi_{x,z}$ and $\Psi_{x,z}$. }
\label{fig:local_exp}
\end{figure}
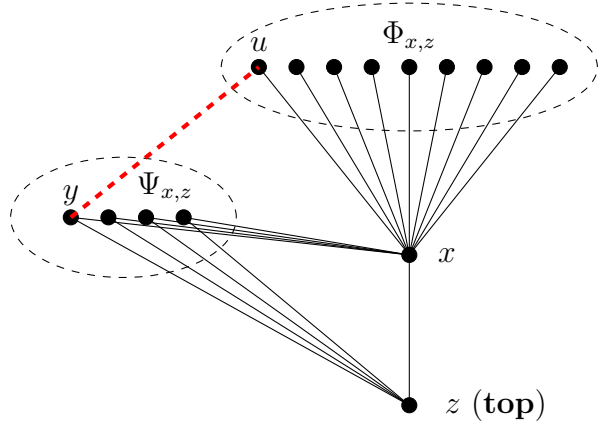

To complete the details, we place $u$ first in the ordering of $\Phi_{x,z}$, and $y$ first in $\Psi_{u,x}$. We interpret each subword $F_iB$ for $i=1,\ldots,a-b-1$ (including $F_1B$) as before. However, we allow as well the subword $F_{1^{\ast}} B$ to which we associate the transitions $x\to u\to y$. The same applies to subwords $F_iS_j$ which we interpret as usual. However, we forbid the subword $F_1S_1$ to avoid overcounting the walk $x\to u \to y$. 

This change affects the census in Theorem \ref{ABTB}. A subword $FB$ has now $a-b$ rather $a-b-1$ realizations, whereas for $FS$ the count goes down from $(a-b-1)b$ to $(a-b-1)b-1$. As the next calculation shows, the gain outweighs the loss, yielding a better lower bound on $\lambda_2$.

Clearly there are $\left( \beta ^{2\beta}(1-2\beta)^{1-2\beta} \right)^{-t(1+o_t(1))}$ length $t$ words in the alphabet $\{F,B,S\}$ with $\beta t$ letters $F$ and $B$ and $(1-2\beta)t$ letters $S$. Standard concentration-of-measure inequalities show that with a proper choice of the $o_t(1)$ terms, the same asymptotic counts remains even if we insist that:
\begin{itemize}
\item 
The Catalan condition for $F$ and $B$ holds
\item
Every pair of consecutive letters appears the "right" number of times. E.g., the number of $FB, FS, SS$ subwords is $(1+o_t(1))\beta^2t$, resp.\ $(1+o_t(1))\beta(1-2\beta)t$, and $(1+o_t(1))(1-2\beta)^2t$, etc.
\end{itemize}
For every such word, we compute the number of permissible ways to index the $F$-steps and the $S$-steps as following:  

\begin{itemize}
\item 
A letter $F$ that is followed by an $F$ can be indexed in $a-b-1$ ways.
\item A letter $S$ that is not preceded by an $F$ can be indexed in $b$ ways.
\item A pair of consecutive letters $FB$ can be indexed in $a-b$ ways.
\item A pair of consecutive letters $FS$ can be indexed in $(a-b-1)b-1$ ways.
\end{itemize} 

In summary, we seek to maximize
\[
\begin{split}
\mathcal{H}(\beta,\beta,1-2\beta)&+\beta^2 \log(a-b-1) + (1-2\beta)(1-\beta)\log b\\
&+\beta^2\log(a-b)+(\beta-2\beta^2)\log(b(a-b-1)-1).
\end{split}
\]
Write $\log(b(a-b-1)-1)=\log b+\log(a-b-1)+\log(1-\frac{1}{b(a-b-1)})$ and $\log(a-b)=\log(a-b-1)+\log(1+\frac{1}{a-b-1})$ to conclude that instead of the analysis of Equation (\ref{disco}) we now seek $\beta$ that maximizes
\begin{equation}\label{delta}
S(\beta,a,b)=\mathcal{H}(\beta,\beta,1-2\beta)+\beta \log(a-b-1) + (1-2\beta)\log b + \Delta,
\end{equation}
where
\[
\Delta=\beta^2\log\left (1+\frac{1}{a-b-1}\right )+\beta(1-2\beta)\log\left (1-\frac{1}{b(a-b-1)}\right).
\]

We now prove that $\max_{\beta}S(\beta,a,b) > \log(b+2\sqrt{a-b-1})$ whenever $a-b\ge 3$ and $b\ge 2$. The proof for complementary parameter range follows from by observing that $G$ is necessarily comprised of disjoint copies of the same graph $H$ and is therefore not even connected. If $b=0,1$ and $\delta>0$, then $H$ is a triangle. When $a-b=1,2$ the same holds with $H=K_{a+1}$ and $H=(K_{a+2}$ minus a perfect matching) respectively.

We denote $c=\sqrt{a-b-1}$. The values of $$\frac{(b+2c)^2}{\log(e)}\left(\max_{\beta}S(\beta,a,b)-\log(b+2c)\right)$$ for small $a$ and $b$'s are shown in the following table:
\vspace{0.2in}
\begin{center}
\begin{tabular}{ |c||c|c|c|c|c|c|c|}
 \hline
 \multicolumn{8}{|c|}{$\frac{(b+2c)^2}{\log(e)}\cdot\max(S(\beta,a,b)-\log(b+2c))$} \\
 \hline
 $a-b-1$ and $b$ & $2$ & $3$ & $4$ & $5$ & $6$ & $7$ & $8$ \\
 \hline
 $2$ & $0.062$ & $0.08$  & $0.088$ & $0.092$ & $0.094$ & $0.096$ & $0.097$\\
 $3$ & $0.281$ & $0.287$ & $0.29$  & $0.29$  & $0.291$ & $0.291$ & $0.291$\\
 $4$ & $0.397$ & $0.401$ & $0.402$ & $0.402$ & $0.402$ & $0.401$ & $0.401$\\
 $5$ & $0.472$ & $0.474$ & $0.475$ & $0.475$ & $0.475$ & $0.474$ & $0.474$\\
 $6$ & $0.525$ & $0.527$ & $0.527$ & $0.527$ & $0.527$ & $0.527$ & $0.526$\\
 $7$ & $0.565$ & $0.567$ & $0.567$ & $0.567$ & $0.567$ & $0.566$ & $0.566$\\
 $8$ & $0.597$ & $0.598$ & $0.599$ & $0.599$ & $0.598$ & $0.598$ & $0.598$\\
 \hline
\end{tabular}    
\end{center}
\vspace{0.2in}
\noindent Consequently, in proving our statement we can ignore the case where both $b$ and $c$ are small. We do not find a closed form expression for $\beta$ that maximizes (\ref{delta}).
Instead we let $\beta:=\frac{c}{b+2c}$, and show that $S(\frac{c}{b+2c},a,b)>\log(b+2c)$.
With this choice of $\beta$ there holds $\frac{1-2\beta}{\beta}= \frac{b}{c}$, so that

\[
\Delta=\frac{\beta^2}{c}\left(c\cdot\log(1+\frac{1}{c^2})+b\cdot\log(1-\frac{1}{b\cdot c^2}) \right).
\]
It is easily seen that for $c>0$ fixed, this expression is an increasing function of $b$, whence it suffices to verify that $\Delta > 0$ when $b=2$. Using Taylor expansion it is easily verified that this inequality holds already for $c>1.5$. The same analysis yields that for every $b\ge 2$ and large $c$ there holds $\Delta = \frac{\log(e)-o_c(1)}{(b+2c)^2}$. It follows that if all $G_v$ are connected then
\[
\lambda_2(G) \geq \left( b+2\sqrt{a-b-1}\right)\left(1+\Omega\left(\frac{\log e}{(b+2c)^2}\right)\right)-o_n(1) 
\]
as claimed.

We turn to consider what happens when the graphs $G_v$ expand. In this case, for every two adjacent vertices $x, z$ there are some edges between the sets $\Psi_{x,z}$ and $\Phi_{x,z}$, where, as above, $\Psi_{x,z}:= G_x\cap G_z$, and $\Phi_{x,z}:= G_x\setminus (G_z\cup\{z\})$. 
Let $R$ be the least number of such edges over all $xz\in E(G)$. Hence, by the definition of edge expansion, $R\geq\min(b+1,a-b-1)\cdot \delta$.
Under the assumption that all $G_v$ are connected we pick one edge $uy$ with $u\in \Phi_{x,z}$ and $y\in\Psi_{x,z}$ and create a special forward step denoted by $F_{1^{\ast}}$. We interpret the subword $F_{1^{\ast}}B$ as an instruction to move $x\to u\to y$ and maintaining $z$ on the top of the stack. In addition, we forbid the subword $F_iS_j$ where $u$ is the $i$-th vertex in $\Phi_{x,z}$ and $y$ is the $j$-th vertex in $\Psi_{u,x}$. In the present context we can likewise consider some $r\le R$ edges $u_k y_k$, $k=1,...,r$, with $u_k\in \Phi_{x,z}, y_k\in\Psi_{x,z}$. Associated with them we create $r$ types of forward steps called $F_{1^{\ast}}, F_{2^{\ast}},\ldots,F_{r^{\ast}}$ and associate with the subword $F_{k^{\ast}}B$ the move $x\to u_k\to y_k$, while $z$ stays on the top of the stack. In addition, we forbid subwords of the form $F_{i_k}S_{j_k}$, where $u_k$ is the $i_k$-th vertex in $\Phi_{x,z}$ and $y_k$ is the $j_k$-th vertex in $\Psi_{u_k,x}$. This works for any choice of $r\le R$ such edges. Now a pair of consecutive letters $FB$ can be indexed in $a-b+r-1$ ways but a pair of consecutive letters $FS$ can be indexed only in $(a-b-1)b-r$ ways. This yields the same maximization problem of (\ref{delta}) with the correction term
\[
\Delta=\beta^2\log\left (1+\frac{r}{a-b-1}\right )+\beta(1-2\beta)\log\left (1-\frac{r}{b(a-b-1)}\right).
\]

By letting $\beta:=c/(b+2c)$ as before this reformulates as
\[
\Delta=\frac{\beta^2}{c}\left(c\cdot \log(1+\frac{r}{c^2}) +b\cdot \log(1-\frac{r}{bc^2})\right).
\]
Straightforward calculations show that the value of $r$ that maximizes this expression is $\frac{b(c^3-c^2)}{b+c}$. So, we let $r:=\min\left(R, \lceil\frac{b(c^3-c^2)}{b+c}\rceil\right)$. In this case, when $c$ is large, we get $\Delta=\frac{r\log(e)-o_c(1)}{(b+2c)^2}$ and 
\[
    \lambda_2(G) \geq \left( b+2\sqrt{a-b-1}\right)\left(1+\Omega\left(\frac{r\log(e)}{(b+2c)^2}\right)\right)-o_n(1) 
\]
completing the proof. 
\end{proof}

\begin{rem}
\rm
If $\delta<\frac{b(a-b-1)(\sqrt{a-b-1}-1)}{(b+\sqrt{a-b-1})\min(b+1,a-b-1)}$, then
$\epsilon$ increases with $\delta$. Note that if $a\geq  b^2+5b+5$, this restriction on $\delta$ is vacuous and $\epsilon$ is always increasing, since the edge expansion of a $b$-regular graph cannot exceed $\frac{b}{2}$. 
\end{rem}

\section{The Polygraph}\label{sec:polyg}
In this section we provide a construction of infinite families of $(a,b)$-regular graphs with good local and global expansion properties. This construction is strictly combinatorial and resembles certain graph products such as Cartesian powers and tensor powers. \Tref{Poly_are_exp} is proved at the end of this section.

The following terminology is used throughout: Let $S$ be a multiset of $k$ non-negative integers and multiplicities $m_1,...,m_k$. We use the notation $S=[ l_1,\ldots,l_m]$, where $m=\sum m_i$. Let $G$ be a $d$-regular graph with $d\geq 3$ and girth bigger than $3\max(S)$. Let $\rho$ be the graph metric of $G$, i.e.\ for two vertices $x,y\in V(G)$, $\rho(x,y)$ is the length of a shortest path between them.
\begin{defin}
\rm
\emph{The Polygraph} $G_S=(V_S,E_S)$ has vertex set $V_S=V(G)^m$, and two vertices $\bar{x}=(x_1,\ldots,x_m)$ and $\bar{y}=(y_1,\ldots,y_m)$ in  $V_S$ are neighbors if and only if  $[\rho(x_1,y_1),...,\rho(x_m,y_m)]=S$ as multisets.
\end{defin}
\noindent
The \emph{distance profile} of any two vertices  $\bar{x}$ and $\bar{y}$ in  $V_S$ is $$\bar{\rho}(\bar{x},\bar{y})=(\rho(x_1,y_1),...,\rho(x_m,y_m)).$$ 
Thus, $\bar x \bar y \in E_S$ given $\bar \rho (\bar x, \bar y)=S$ as multisets.
Conversely, if $\bar\rho(\bar x,\bar y)=(d_1,...,d_m)$, then $\bar x \bar y\in E_{Z}$, where $Z=[d_1,...,d_m]$.

\begin{rem} \label{TenRem}
\rm
When $S'=[1,0,...,0]$, the polygraph $G_{S'}$ coincides with $G^{\square m}$, the $m$-th Cartesian power of $G$.
If $S=[ l_1,\ldots,l_m]$ and $N=\sum_{i=1}^m l_i$, then $G_S$ embeds in $G^{\square m}$, where every edge of $G_S$ can be mapped to a (non-unique) length $N$ non-backtracking path. Some of the proofs below involve polygraphs with the same graph $G$ and two distinct multisets $S_1, S_2$. In such situations, it is useful to embed  both $G_{S_1}$ and $G_{S_2}$ into the same $G^{\square m}$.
\end{rem}

\begin{claim} \label{regularity_GS}
The polygraph $G_S$ is $(a_S, b_S)$-regular, where $a_S$ and $b_S$ depend only on $S$.
\end{claim}
\begin{proof}
 Clearly,
\[
a_{S}=\frac{m!}{m_1!\cdot...\cdot m_k!}\cdot \prod_{j=1}^m\lfloor d(d-1)^{l_j-1}\rfloor.
\]
\noindent
The first term counts the distinct rearrangements of $S$, and for $t<\text{girth}(G)$ there are exactly $\lfloor d(d-1)^{t -1}\rfloor$ vertices at distance $t$ from any vertex in $G$. The floor is relevant only when $l_j=0$, in which case $\lfloor d(d-1)^{l_j -1}\rfloor=\lfloor\frac{d}{d-1}\rfloor=1$

We next consider $b_S$ (which may be zero). Given a vertex $\bar x=(x_1,...,x_m)$, the balls of radius $\max(S)$ around $x_i$ in $G$ are isomorphic to a ball of the same radius around some fixed vertex $\xi$ in $\mathbb{T}_d$ the $d$-regular tree. Moreover, since $\textrm{girth}(G)>3\max(S)$, if $y_i,z_i\in B_{\max(S)}(x_i)$ have $\rho(y_i,z_i)\leq \max(S)$, then the shortest path between them is contained in the ball $B_{\max(S)}(x_i)$. Thus $(G_S)_{\bar x}$, the link of $\bar{x}$ in $G_S$, is isomorphic to $((\mathbb{T}_d)_S)_{\bar \xi}$, the link of $\bar\xi =(\xi,...,\xi)$ in $(\mathbb{T}_d)_S$. Since our argument does not depend on the choice of $\bar x$, the degree $b_S$ is well-defined if we prove that $((\mathbb{T}_d)_S)_{\bar \xi}$ is regular.

The symmetric group $S_m$ acts naturally on $(\mathbb{T}_d)_S$ by permuting the coordinates. Moreover, the automorphism group $\textrm{Aut}(\mathbb{T}_d)$ acts on each coordinate separately. Thus, for $\textrm{Stab}(\xi)\leq \textrm{Aut}(\mathbb{T}_d)$ the stabilizer of $\xi\in\mathbb{T}_d$, the group $G=\textrm{Stab}(\xi) \wr S_m=(\prod_{i=1}^m \textrm{Stab}(\xi)) \rtimes S_m$  is a subgroup of the stabilizer of $\bar{\xi}$ in $\textrm{Aut}((\mathbb{T}_d)_S)$. Hence  $G$ acts on the link of $\bar{\xi}$. Since $\textrm{Stab}(\xi)$ acts transitively on spheres around $\xi$, and $S_m$ acts transitively on all distance profiles, we can conclude that $G$ acts transitively on the vertices of the link $((\mathbb{T}_d)_S)_{\bar{\xi}}$. Clearly, a graph with a transitive automorphism group is regular.
\end{proof}

We provide a closed formula for $b_S$ in the appendix. It is based on the following arithmetic criterion on $S$.

\begin{claim}\label{b_S=0}
With the above notation, $b_S>0$ iff there is a $3\times m$ matrix, every row of which is comprised of the integers $l_1,\ldots,l_m$ in some order, where every column has an even sum and satisfies the triangle inequality.
\end{claim}
\begin{proof}
Necessity: Assume $\bar x, \bar y, \bar z$ is a triangle in $G_S$. Since $G$'s girth is large, the three geodesic paths connecting $x_i, y_i$ and $z_i$ form a tree, whence the sum of their lengths is even and the three lengths satisfy the triangle inequality. Hence, one can construct a matrix whose rows are the distance profiles of the edges of $\bar x\bar y, \bar x\bar z, \bar y\bar z$, meaning  $\bar \rho (\bar x, \bar y),\ \bar \rho (\bar x, \bar z)$ and $\bar \rho (\bar y, \bar z)$. \\
Sufficiency is not hard either: given three integers smaller than $\text{girth}(G)$ with even sum that satisfy the triangle inequality, there are three vertices in $G$ the distances between which are these three integers. This allows us to construct $\bar x, \bar y, \bar z$ one coordinate at a time.
\end{proof}

\begin{rem}\label{even_sum}
\rm
For some of our applications we restrict ourselves to the case where $N=\sum_{i=1}^m l_i$ is even and $G$ is not bipartite. The reason is that by the claim above, $b_S=0$ when $N$ is odd. Also, if $N$ is even and $G$ is bipartite, then $G_S$ is disconnected.
\end{rem}

We suspect that it is computationally hard to decide whether the condition in Claim \ref{b_S=0} can be satisfied for a given $S$. This is no problem for small $m$, thus here is the solution for $m=3$:

\begin{claim}\label{concond}
Let $S$ be the multiset of integers $p, q, r\ge 0$. Then $b_S\ne 0$ if and only if (i) $p,q,r$ are all even, or (ii) Their sum is even and they satisfy the triangle inequality.
\end{claim}

\begin{proof}
In case (i) this is shown by the matrix
\[
\begin{bmatrix}
p & q & r\\
p & q & r\\
p & q & r
\end{bmatrix}.
\]
In case (ii) this is done by the matrix
\[
\begin{bmatrix}
p & q & r\\
q & r & p\\
r & p & q
\end{bmatrix}.
\]

On the other hand, assume $b_S\ne 0$. Thus, since the sum of each column of the table is even, the sum of all entries which is $3p+3q+3r$ is even and therefore $p+q+r$ must be even. If all of them are even, that is case (i). Otherwise, since the even number must appear in every column, there is a column with all three numbers and they thus satisfy the triangle inequality.
\end{proof}

\subsection{Non-backtracking paths}
Let $G$ be a graph with adjacency matrix $A_G$ and let $A_G^{(t)}$ be the matrix whose $(i,j)$-th entry is the number of length-$t$ non-backtracking paths between vertices $i$ and $j$ in $G$.
We also view $A_G^{(t)}$ as the adjacency matrix of a multigraph $G^{(t)}$. 

When $G$ is regular, $G^{(t)}$ is $d(d-1)^{t-1}$-regular. In the regular case these matrices satisfy the following recursion:
\[
\begin{split}
A_G^{(1)}&=A_G, \\ 
A_G^{(2)}&=A_G^2-d I_{n},\\ 
A_G^{(t+1)}&=A_G A_G^{(t)}-(d-1) A_G^{(t-1)}.
\end{split}
\]
The so-called Geronimus polynomials $p^{(t)}(\cdot)$ satisfy $p^{(t)}(A_G)=A_G^{(t)}$. For more on this see \cite{sole}, \cite{ABLS} or \cite{DSV}.

\begin{lem} \label{Lem_aperiodicity}
If $G$ is a connected, non-bipartite graph with minimum vertex degree at least $3$, then for every $t$, $G^{(t)}$ is a connected non-bipartite graph.
\end{lem}

\begin{proof}
To show that $G^{(t)}$ is non-bipartite, let $A$ be the vertex set of an odd-length cycle in $G$. Then the subgraph of $G^{(t)}$ spanned by $A$ contains an odd cycle.

Clearly $A$ spans a connected subgraph of $G^{(2)}$, and every vertex can reach $A$ by a path in $G^{(2)}$. To prove that $G^{(t)}$ is connected for $t\ge 3$ we show that if $uv\in E(G^{(2)})$, then $u, v$ are in the same component of $G^{(t)}$. Since $uv\in E(G^{(2)})$, they have a common neighbor $w$ in $E(G)$. Since $d_G(w)\ge 3$, there is a length-$(t-1)$ path $P$ starting at $w$ whose first step is neither to $u$ nor $v$. The claim follows, since $P$'s last vertex is a neighbor of both $u$ and $v$ in $G^{(t)}$. 
\end{proof}

\subsection{Connectivity and Spectral analysis of $G_S$}
Let $\Omega=\binom{[m]}{m_1,\ldots,m_k}$ be the set of rearrangements $\omega=(\omega_1,\ldots,\omega_m)$ of $S=[ l_1,\ldots,l_m]$. Thus 
\[
A_{G_S}=\sum_{\omega\in \Omega} \bigotimes_{j=1}^m A_G^{(\omega_j)}
\]
where $\otimes$ is the Kronecker tensor product.

If $v$ is an eigenvector of $A$ with eigenvalue $\lambda$, likewise for $v', A'$ and $\lambda'$, then $v\otimes v'$ is an eigenvector of $A\otimes A'$ with eigenvalue $\lambda\lambda'$. Also, $A_G^{(t)}$ and $A_G$ have the same eigenvectors, since $A_G^{(t)}$ is a polynomial in $A_G$. It follows that every eigenvector of $A_{G_S}$ has the form $v_1\otimes  ...\otimes v_m$ where each $v_i$ is an eigenvector of $A_G$'s. Moreover, by going through all such choices of $v_1,\ldots,v_m$ we obtain the full list of eigenvectors. The eigenvalue of $v_1\otimes  ...\otimes v_m$ is 
\[
\chi (\lambda_1,...,\lambda_m)=\chi_{S}(\lambda_1,...,\lambda_m) = \sum_{\omega\in \Omega}\prod_{j=1}^m p^{(\omega_j)}(\lambda_j)
\]
where $\lambda_i$ the eigenvalue of $v_i$, and $p^{(t)}(x)$ is the $t$-th Geronimus polynomial mentioned above. We obtain all the eigenvalues of $A_{G_{S}}$ by evaluating the symmetric polynomial $\chi_{S}$ on all $m$-tuples of eigenvalues $(\lambda_1,...,\lambda_m)$.

\begin{claim}\label{G_S-connected}
Let $G$ be a connected non-bipartite $d$-regular graph and $S$ a multiset of non-negative integers, not all zero. Then $G_S$ is connected and non-bipartite.
\end{claim}

\begin{proof}
It suffices to prove that for every choice of $(\lambda_1,...,\lambda_m)$, eigenvalues of $G$, not all $d$, we have
\[
\vert\chi(\lambda_1,...,\lambda_m)\vert < a_S=\chi(d,...,d).
\]
Since $p^{(t)}(d)$ is the common degree of the vertices of $A_G^{(d)}$, it must be the largest eigenvalue of this matrix. Thus, for every eigenvalue $\lambda_j$ of $A$ we have $\vert p^{(t)}(\lambda_j)\vert\leq p^{(t)}(d)$. By \Lref{Lem_aperiodicity}, when $t>0$ and $\lambda_i\neq d$ we have also $\vert p^{(t)}(\lambda_j)\vert< p^{(t)}(d)$. Now
\[
\begin{split}
\vert\chi(\lambda_1,...,\lambda_m)\vert &\leq  \sum_{\omega\in \Omega}\prod_{j=1}^m \vert p^{(\omega_j)}(\lambda_j)\vert <\sum_{\omega\in \Omega}\prod_{j=1}^m p^{(\omega_j)}(d)
\end{split}
\]
and the inequality is strict since for some $j$ and $\omega$ we have both $\lambda_j\neq d$ and $\omega_j\neq 0$.
\end{proof}
 
We now seek stronger bounds on $G_S$'s spectral gap under appropriate assumptions on $G$. The following lemma is used in our analysis of geometric overlap.
\begin{lem} \label{Geroni}
Let $\alpha\in(-2\sqrt{d-1}, 2\sqrt{d-1})$ and $|\beta|>2\sqrt{2}\sqrt{d-1}$. Then $|p^{(t)}(\beta)|>|p^{(t)}(\alpha)|$, where $p^{(t)}$ is the $t$-th Geronimus Polynomial. Also, if $|x| \ge 2\sqrt{d-1}$, then $|p^{(t)}(x)| < |x|^t$.
\end{lem}
\begin{proof}
All we need to know about $p^{(t)}$ is that all its roots are real, they reside in $(-2\sqrt{d-1}, 2\sqrt{d-1})$ and come in pairs $\pm \rho$  plus a root at zero for odd $t$ (e.g., \cite{lmn}). Thus, $p^{(t)}(x)=x^{t\bmod 2}\prod (x^2-\rho^2)$. To prove the claim, compare between $p^{(t)}(\alpha)$ and $p^{(t)}(\beta)$ term by term, i.e., verify that $|\beta^2-\rho^2| > |\alpha^2-\rho^2|$ whenever $\alpha, \rho \in (-2\sqrt{d-1}, 2\sqrt{d-1})$ and $|\beta|>2\sqrt{2}\sqrt{d-1}$.\\
The equality $p^{(t)}(x)=x^{t\bmod 2}\prod (x^2-\rho^2)$ readily yields
the second claim.
\end{proof}
 
Therefore, if $\lambda(G)\geq 2\sqrt{2}\sqrt{d-1}$, then we have
\[
 \begin{split}
 \lambda(G_S)&=\chi(\lambda(G),d,...,d)\\
 &=\sum_{\omega\in \Omega}p^{(\omega_1)}(\lambda(G))\prod_{j=2}^m p^{(\omega_j)}(d)\\
 &=\sum_{i=1}^m p^{(l_i)}(\lambda(G))\cdot \frac{(m-1)!\cdot m_i}{m_1!\cdot \dots \cdot m_k!}\prod_{j\neq i}\lfloor d(d-1)^{l_j-1}\rfloor.
 \end{split}
 \]
 This formula is nice, but we are interested in  a more practical bound on $\lambda(G)$. Thus the following:
 
 \begin{prop} \label{lambdaGS}
 Let $G$ be a $d$-regular graph and let $S$ be a multiset of non-negative integers, the smallest of which is $s$, whose sum is $N$. Then, letting $\mu=\max(\lambda(G),2\sqrt{d-1})$, we have  $$\lambda(G_S)\le \binom{m}{m_1,...,m_k} \mu^s d^{k-1}(d-1)^{N-k-s+1}.$$
\end{prop}

\begin{proof}
Similar to the previous calculation using the second statement of \Lref{Geroni}.

\end{proof}

\subsection{ The Triangles of $G_S$} 
\label{Triangles}
We seek a structural description of the triangles in $G_S$, where $S=[l_1,...,l_m]$. We may assume, of course, that $N=\sum_{i=1}^m l_i$ is even, for otherwise (\Clref{b_S=0}) $G_S$ is triangle-free. Two facts should be kept in mind. Throughout the paper we always posit that $\text{girth}(G)$ is "sufficiently large", whence $G$'s relevant local structure coincides with that of $\mathbb{T}_d$, the $d$-regular tree. Also, by Remark \ref{TenRem}, $G_S$ embeds into $G^{\square m}$ by mapping edges to geodesic length-$N$ paths. 

A simple but useful property of trees $T$ is that every triple of vertices $x,y,z\in V(T)$ has a unique \emph{center}. This is the vertex $c$ that minimizes $\rho_T(x,\cdot)+\rho_T(y,\cdot)+\rho_T(z,\cdot)$. It is also characterized by the condition $\rho_T(x,c)+\rho_T(y,c)+\rho_T(z,c)=(\rho_T(x,y)+\rho_T(y,z)+\rho_T(z,x))/2$: By the triangle inequality 
\[
2(\rho_T(x,c)+\rho_T(y,c)+\rho_T(z,c))
\leq \rho_T(x,y)+\rho_T(z,x)+\rho_T(y,z).
\] 
On the other hand, consider the finite tree induced by the three vertices $x,y,z$, namely, the triangle $\triangle xyz$. It has a single vertex where its three branches meet. This may coincide with one of the original vertices of the triangle, see the two cases in Figure \ref{Triangles_in_a_tree}. This meeting point is $c$, for which the above clearly holds with equality.
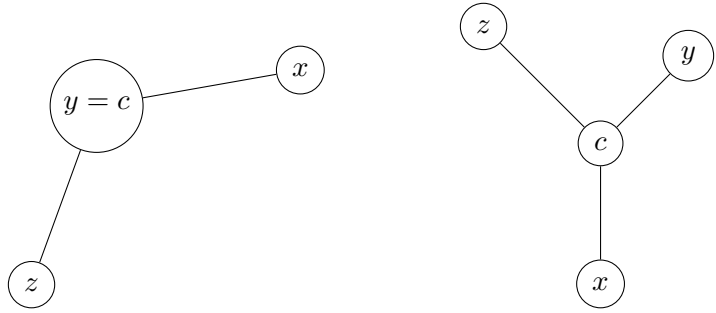
\begin{figure}[]\label{Triangles_in_a_tree}
    \centering
    \caption{Possible triangles $\triangle xyz$ in a tree. \textbf{The edges represent paths of some length}.}
    \label{fig:TriG123}
\begin{center}

\begin{tikzpicture}[scale=1.1]
\node[draw,shape=circle] (w0) at (0:0) {$y=c$};
\node[draw,shape=circle] (w1) at (10:2.5) {$x$};
\node[draw,shape=circle] (w3) at (250:2.3)  {$z$};
\draw (w0) -- (w1)
(w0) -- (w3);
\end{tikzpicture}
\qquad  \qquad \ 
\begin{tikzpicture}[scale=1.1]
\node[draw,shape=circle] (w31) at (0:0) {$c$};
\node[draw,shape=circle] (u3) at (270:1.7) {$x$};
\node[draw,shape=circle] (x3) at (45:1.5) {$y$};
\node[draw,shape=circle] (v3) at (135:2) {$z$};
\draw (w31) -- (u3)
(w31) -- (x3)
(w31) -- (v3);
\end{tikzpicture}
\end{center}
\end{figure}

More generally, we define the \emph{center} $\bar c=(c_1,...,c_m)$ of a triple $\bar x\bar y \bar z$ in $(\mathbb{T}_d)_S$, by taking, for $1\le i\le m$ the vertex $c_i$ as the center of $x_i,y_i,z_i$ in $\mathbb{T}_d$.
Let $\delta$ be the metric of $\mathbb{T}_d^{\square m}$, i.e., the coordinate-wise sum of $\rho_T$ distances. We say that $\bar w$ is {\em a midpoint} of $\bar u, \bar v$ (all in $(\mathbb{T}_d)_S$) if $ \delta(\bar u, \bar w)=\delta(\bar v, \bar w)= \delta(\bar u,\bar v)/2$.

\begin{claim}
The center $\bar c$ of a triangle $\bar x, \bar y, \bar z$ in $(\mathbb{T}_d)_S$ is a midpoint of each of its three edges.
\end{claim}

\begin{proof}
For each $1\le i\le m$, the vertex $c_i$ is the center of the triple $x_i, y_i, z_i$ in $\mathbb{T}_d$, whence $$\rho_T(x_i,c_i)+\rho_T(y_i,c_i)+\rho_T(z_i,c_i)=(\rho_T(x_i,y_i)+\rho_T(y_i,z_i)+\rho_T(z_i,x_i))/2.$$ Summing over $i$, we get 
\[
\delta(\bar c,\bar x)+ \delta(\bar c,\bar y)+\delta(\bar c,\bar z )=\frac{1}{2}( \delta(\bar x,\bar y)+ \delta(\bar x,\bar z)+\delta(\bar y,\bar z))=\frac{3}{2}N.
\]
Coordinatewise, $\rho_T(z_i,c_i)+\rho_T(c_i,y_i)=\rho_T(z_i.y_i)$, since $c_i$ is on the path between $z_i$ and $y_i$ in the tree. Hence $\delta(\bar z,\bar c)+\delta (\bar c,\bar y)=\delta(\bar y, \bar z)$ and thus $\delta(\bar c , \bar x)=\frac{1}{2}N$.
\end{proof}

\begin{rem}
Note that a related question arises in the proof of \Tref{GeoOver} concerning the geometric overlapping properties of $G_S$. That is, given an edge $\bar x\bar y \in E_S$ whose midpoint is  $\mathfrak{ m}$, how many triangles with center $\mathfrak{m}$ are incident to $\bar x \bar y $?
\end{rem}

\subsection{Local Connectivity of $G_S$}

We start with a necessary condition for $L=(G_S)_v\cong ((\mathbb{T}_d)_S)_{(\xi,...,\xi)}$ to be connected. Here we denote the connected components of $\mathbb{T}_d\setminus\{\xi\}$ by $B_1,...,B_d$. Figure \ref{Tree_With_Marks} can help the reader follow the calculations in this subsection.
\begin{lem}\label{necnec}
If $L$ is connected then either (i) $0\in S$, or (ii) there is a positive $s\in S$ such that $2s\in S$ as well, or (iii) there are three distinct $s,s',s''\in S$ satisfying $s''=s+s'$. 
\end{lem}
\begin{proof}
Here $\rho$ is the metric of $\mathbb{T}_d$. If there is some $\bar{u}\in L$ with $u_1=\xi$, then necessarily $0 \in S$. Otherwise there must be an edge $\bar{u}\bar{w} \in E(L)$ with, say, $u_1\in B_1, w_1\in B_2$. But then $\rho(u_1,w_1)=\rho(u_1, \xi)+\rho(w_1, \xi)$. If $\rho(u_1, \xi)=\rho(w_1, \xi)$, we are in case (ii) and if they differ, case (iii) emerges.
\end{proof}

The following two claims give a necessary and sufficient condition for connectivity for $m=2,3$.

\begin{claim}\label{pi}
Let $S=[p,q]$, where $q\ge p$ and $q>0$. Then, $L$ is connected if and only if $p$ is even and $q=2p$.
\end{claim}

\begin{proof}
If $L$ is connected, then $b_S > 0$, and Claim \ref{b_S=0} implies that $p$ and $q$ are even. By \Lref{necnec}, either $q=2p$ or $p=0$. The latter case is ruled out by Claim \ref{b_S=0}, since the relevant matrix is necessarily: 
\[
\begin{bmatrix}
0 & q \\
0 & q \\
0 & q 
\end{bmatrix}
\]
whence the first coordinate cannot change and $L$ is not connected. This shows the necessity of the condition $q=2p$.

To prove sufficiency we need to exhibit a path between any two vertices in $L$, say $(u_1,u_2)\in B_{\alpha_1}\times B_{\alpha_2}$ and a vertex in $B_{\beta_1}\times B_{\beta_2}$. We select a {\emph{pivot}} vertex in $B_{\gamma_1}\times B_{\gamma_2}$ with $\gamma_1 \neq \alpha_1, \beta_1$ and $\gamma_2\neq \alpha_2, \beta_2$, which they can both reach. Such $\gamma_1, \gamma_2$ exist, since $d\ge 3$. So suppose ${\bar \rho}((u_1,u_2),(\xi,\xi))=(p,2p)$, and consider the following path:
\begin{equation} \label{path2}
    (p,2p) \xrightarrow{(2p,p)} (-p,2p) \xrightarrow{(2p,p)} (-2p,p) \xrightarrow{(p,2p)} (-2p,-p) \xrightarrow{(p,2p)} (-p,-2p).
\end{equation}

Here we use the following notation: A value $x>0$ in the $i$-th coordinate represents a vertex in $B_{\alpha_i}$ at distance $x$ from $\xi$, whereas $-x$ represents a vertex in some $B_{\delta}$ with $\delta\neq\alpha_i$ at distance $x$ from $\xi$. A pair $(r, s)$ above an arrow indicates that we move $r$ away from the vertex in the first coordinate, and $s$ away in the second coordinate. For illustration refer to Figure \ref{Tree_With_Marks}, with $p=2, d=3$, to follow such a path in $((\mathbb{T}_3)_{[2,4]})_{\bar \xi}$.
\[
(x_{11},x_{1111}) \xrightarrow{(4,2)} (x_{22},x_{1112}) \xrightarrow{(4,2)} (x_{2111},x_{11}) \xrightarrow{(2,4)} (x_{2112},x_{22}) \xrightarrow{(2,4)} (x_{21},x_{2111}).
\]
By the symmetry of the tree, (\ref{path2}) provides a path between $(u_1,u_2)\in B_{\alpha_1}\times B_{\alpha_2}$ and \textbf{every} vertex in  $B_{\gamma_1}\times B_{\gamma_2}$, in particular to the chosen pivot vertex. Consequently, $L$ is connected.
\end{proof}
\begin{figure}[]
    \centering
    \caption{The ball of radius $4$ around the vertex {\color{blue} $\xi$} in $\mathbb{T}_3$. In $\mathbb{T}_3 \setminus {\color{blue} \xi}$ the connected component of $x_i$ is $B_i$. Thus {\color{brown}\bf $B_1$}, {\color{purple} \bf $B_2$} and $B_3$ are comprised of the {\color{brown}\bf brown}, {\color{purple}\bf purple} and the black vertices, respectively. Note that $\left({\color{blue}\xi},  {\color{brown}x_{1111}},  {\color{brown}x_{1121}}\right)$ is an equilateral triangle with center ${\color{brown}x_{11}}$, namely all its edges are of length $4$ and all the paths betwen the vertices go through ${\color{brown}x_{11}}$. This illustrates how, for $p$ even, one can move from a vertex at distance $p$ from the origin to another one of the same distance using a length $p$ path.}
    \label{Tree_With_Marks}
    \label{fig:TriG123}
\begin{center}

\begin{tikzpicture}[scale=1.1]
\node[draw,shape=circle, color=blue] (xi) at (0:0) {$\xi$};
\node[draw,shape=circle, color=brown] (x1) at (0:1) {$x_1$};
\node[draw,shape=circle, color=purple] (x2) at (120:1)  {$x_2$};
\node[draw,shape=circle] (x3) at (240:1)  {$x_3$};
\node[draw,shape=circle, color=brown] (x11) at (20:2.5)  {$x_{11}$};
\node[draw,shape=circle, color=brown] (x12) at (340:2.5)  {$x_{12}$};
\node[draw,shape=circle, color=purple] (x21) at (100:2.5)  {$x_{21}$};
\node[draw,shape=circle, color=purple] (x22) at (140:2.5)  {$x_{22}$};
\node[draw,shape=circle] (x31) at (220:2.5)  {$x_{31}$};
\node[draw,shape=circle] (x32) at (260:2.5)  {$x_{32}$};

\node[draw,shape=circle, color=brown] (x111) at (10:4)  {$x_{111}$};
\node[draw,shape=circle, color=brown] (x121) at (320:4)  {$x_{121}$};
\node[draw,shape=circle, color=purple] (x211) at (90:4)  {$x_{211}$};
\node[draw,shape=circle, color=purple] (x221) at (145:4)  {$x_{221}$};
\node[draw,shape=circle] (x311) at (200:4)  {$x_{311}$};
\node[draw,shape=circle] (x321) at (265:4)  {$x_{321}$};

\node[draw,shape=circle, color=brown] (x112) at (40:4)  {$x_{112}$};
\node[draw,shape=circle, color=brown] (x122) at (350:4)  {$x_{122}$};
\node[draw,shape=circle, color=purple] (x212) at (120:4)  {$x_{212}$};
\node[draw,shape=circle, color=purple] (x222) at (170:4)  {$x_{222}$};
\node[draw,shape=circle] (x312) at (235:4)  {$x_{312}$};
\node[draw,shape=circle] (x322) at (290:4)  {$x_{322}$};

\node[draw,shape=circle, color=brown] (x1111) at (10:6)  {$x_{1111}$};
\node[draw,shape=circle, color=brown] (x1211) at (320:6)  {$x_{1211}$};
\node[draw,shape=circle, color=purple] (x2111) at (75:6)  {$x_{2111}$};
\node[draw,shape=circle, color=purple] (x2211) at (135:6)  {$x_{2211}$};
\node[draw,shape=circle] (x3111) at (193:6)  {$x_{3111}$};
\node[draw,shape=circle] (x3211) at (257:6)  {$x_{3211}$};

\node[draw,shape=circle, color=brown] (x1112) at (25:6)  {$x_{1112}$};
\node[draw,shape=circle, color=brown] (x1212) at (333:6)  {$x_{1212}$};
\node[draw,shape=circle, color=purple] (x2112) at (90:6)  {$x_{2112}$};
\node[draw,shape=circle, color=purple] (x2212) at (147:6)  {$x_{2212}$};
\node[draw,shape=circle] (x3112) at (208:6)  {$x_{3112}$};
\node[draw,shape=circle] (x3212) at (270:6)  {$x_{3212}$};

\node[draw,shape=circle, color=brown] (x1121) at (40:6)  {$x_{1121}$};
\node[draw,shape=circle, color=brown] (x1221) at (345:6)  {$x_{1221}$};
\node[draw,shape=circle, color=purple] (x2121) at (105:6)  {$x_{2121}$};
\node[draw,shape=circle, color=purple] (x2221) at (165:6)  {$x_{2221}$};
\node[draw,shape=circle] (x3121) at (230:6)  {$x_{3121}$};
\node[draw,shape=circle] (x3221) at (290:6)  {$x_{3221}$};

\node[draw,shape=circle, color=brown] (x1122) at (60:6)  {$x_{1122}$};
\node[draw,shape=circle, color=brown] (x1222) at (358:6)  {$x_{1222}$};
\node[draw,shape=circle, color=purple] (x2122) at (120:6)  {$x_{2122}$};
\node[draw,shape=circle, color=purple] (x2222) at (178:6)  {$x_{2222}$};
\node[draw,shape=circle] (x3122) at (242:6)  {$x_{3122}$};
\node[draw,shape=circle] (x3222) at (305:6)  {$x_{3222}$};

\draw (xi) -- (x1)
(xi) -- (x2)
(xi) -- (x3)
(x11) -- (x1)
(x12) -- (x1)
(x21) -- (x2)
(x22) -- (x2)
(x31) -- (x3)
(x32) -- (x3)
(x111) -- (x11)
(x112) -- (x11)
(x211) -- (x21)
(x212) -- (x21)
(x311) -- (x31)
(x312) -- (x31)

(x121) -- (x12)
(x122) -- (x12)
(x221) -- (x22)
(x222) -- (x22)
(x321) -- (x32)
(x322) -- (x32)

(x1111) -- (x111)
(x1112) -- (x111)
(x2111) -- (x211)
(x2112) -- (x211)
(x3111) -- (x311)
(x3112) -- (x311)

(x1121) -- (x112)
(x1122) -- (x112)
(x2121) -- (x212)
(x2122) -- (x212)
(x3121) -- (x312)
(x3122) -- (x312)

(x1211) -- (x121)
(x1212) -- (x121)
(x2211) -- (x221)
(x2212) -- (x221)
(x3211) -- (x321)
(x3212) -- (x321)

(x1221) -- (x122)
(x1222) -- (x122)
(x2221) -- (x222)
(x2222) -- (x222)
(x3221) -- (x322)
(x3222) -- (x322)

;
\end{tikzpicture}
\end{center}
\end{figure}
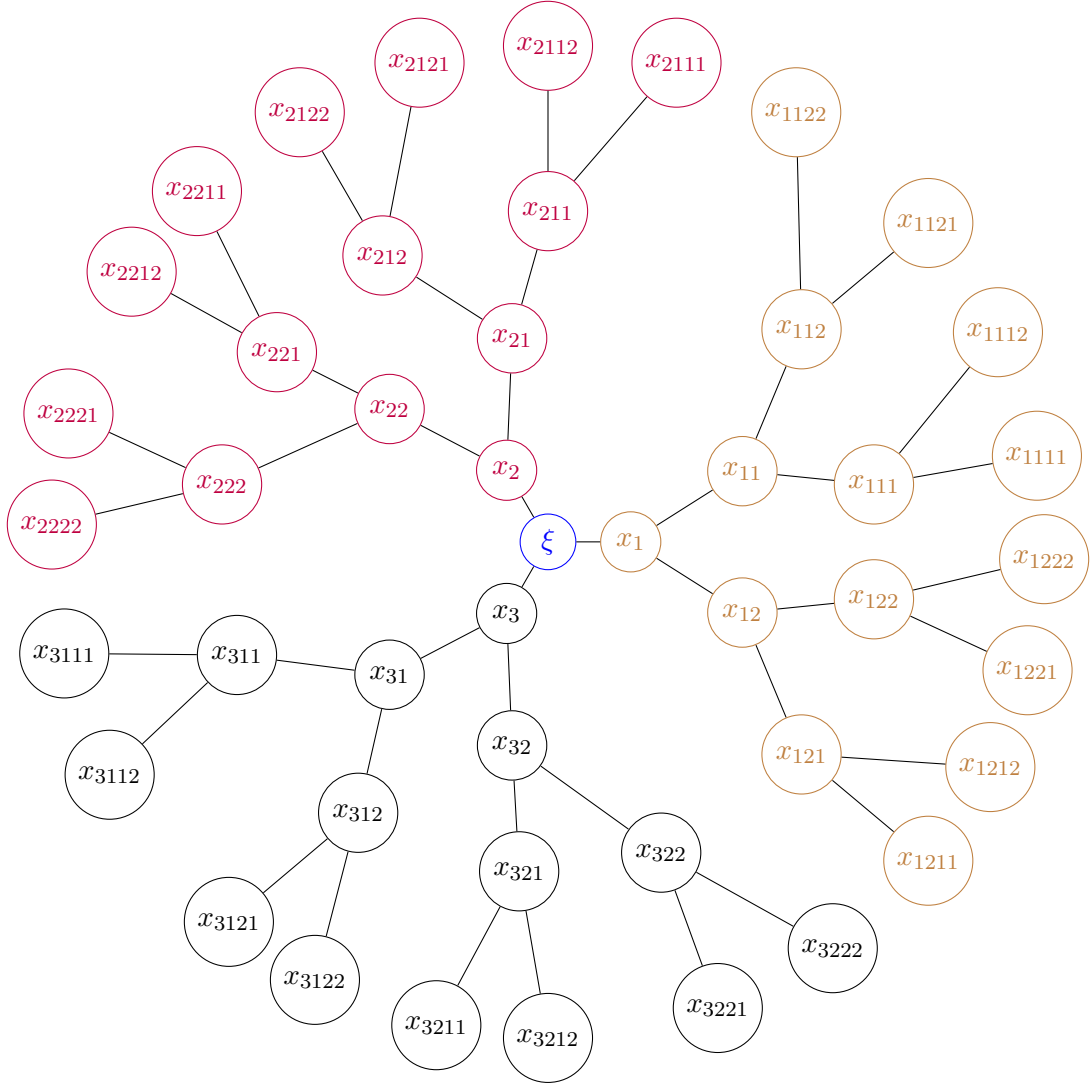

\begin{claim}\label{G_S_v-connected}
Let $S=[p,q,r]$, where $p\leq q\leq r$ and $0<r$. Then $L$ is connected if and only if $p+q+r$ is even and one of the following holds:
\begin{enumerate}
\item $r=p+q$ and either: $p$ is even, or $q$ is even and $2p\geq q$, or $p=q$;
\item $q=2p$ and $r\le p+q$, or $r=2p$;
\item $p,q$ and $r$ are even, $4p\geq 2q \geq r$ and either $r=2p$ or $r=2q$ or $q=2p$.
\end{enumerate}
\end{claim}

\begin{rem}
There is also a criterion for the connectivity of $L$ when all members of $S$ are even for larger $|S|$. We omit the (very technical) details.
\end{rem}

\begin{proof}
By \Lref{necnec}, if $L$ is connected, then one of the following must hold
\begin{equation}\label{condi}
(i)\ r=p+q;\  (ii)\  q=2p;\  (iii)\ r=2p;\  (iv)\ r=2q;\  (v)\ p=0. 
\end{equation}
We start with case $(v)$, i.e., $p=0$ and show that this is included in case $(i)$. Since $b_S>0$, by \Clref{concond} either $q$ and $r$ are both even, or $p+q\geq r$. The second case is just $q=r$ and we refer to the $r=p+q$ case. If $r>q$ and both are even, then the matrix from Claim \ref{b_S=0} must have a column of $0$'s. As in the proof of \Clref{pi}, this implies that $L$ is disconnected.

We turn to show that if $(i)$ holds, i.e., $r=p+q$, then we are in case $(1)$. If $p,q,r$ are all even, this is clear. Otherwise, one of $p,q,r$ is even, call this number $x$ and the other two, called $y,z$ are odd. Since $L$ is connected there must be an edge in $E(L)$ that permutes the distance profile of a vertex in $L$ by an odd permutation. Thus an odd permutation on $S$ should be realizable by a matrix as in \Clref{b_S=0}.
What can such a matrix look like? The element $x$ cannot be a fixed point of this permutation since it must appear exactly once in each column of the matrix. The only odd permutations we can thus realize are the transpositions $x\leftrightarrow y$ and $x\leftrightarrow z$. Consider the case $x\leftrightarrow z$ (the other case is identical). It can be realized only by the following matrices (up to rearrangement of columns):
\[
\begin{bmatrix}
x & y & z\\
y & x & z\\
z & y & x
\end{bmatrix}
\textrm{or}
\begin{bmatrix}
x & y & z\\
z & x & y\\
z & y & x
\end{bmatrix}
\]
Both matrices contain columns comprised of the triplets $x,z,z$ and $x,y,y$, whence $2z\geq x$ and $2y\geq x$. We next work out the correspondence between the multisets $[p, q, r]=[x, y, z]$. If $x=p$, then these two conditions are satisfied automatically. If $x=q$, then $q\leq 2p$. Finally if $x=r$, then $2p\leq p+q=r\leq 2p$ and $r=2p=2q$. In all cases $(1)$ is satisfied.

We next assume that $r\neq p+q$. If $p,q,r$ are not all even, the argument involving odd permutation tables applies and the triangle inequality yields that $r<p+q$. But $p$ is odd by (\ref{condi}), so that one of $q,r$ is even and one is odd. If $r$ is the even one, then the odd permutation table analysis yields $r\leq 2p$. But we are necessarily in case $(iii)$ or $(iv)$ so that $2p\leq r$. It follows that $r=2p$. On the other hand, if $q$ is the even one then $(ii)\ q=2p$ must be satisfied. We have arrived at case $(2)$.

In the only remaining case $p,q,r$ are all even and $r\neq p+q$. We scan cases $(ii), (iii), (iv)$.
\begin{itemize}
    \item If $r=2p$, then clearly $4p=2r\geq 2q\geq 2p= r$;
    \item We claim that $r=2q$ implies $2p\ge q$. Otherwise $p$ can appear only with itself in a column of a permutation table, and $L$ is disconnected;
    \item  Likewise $q=2p$ implies $2q\ge r$. Otherwise $r$ can appear only with at least one more $r$ in the same column, contrary to $L$ being connected.
\end{itemize} 
In all these scenarios we are in case $(3)$.\\

It only remains to verify that under these conditions $L$ is connected. Again we do so by exhibiting a path between any two given vertices via a properly chosen pivot vertex. Say that one vertex with profile $(p,q,r)$ is in $B_{\alpha_1}\times B_{\alpha_2}\times B_{\alpha_3}$ and one in $B_{\beta_1}\times B_{\beta_2}\times B_{\beta_3}$ and its profile is any given permutation $\pi$ of $(p,q,r)$. 

We consider only the case where $p$ is even and $r=p+q$. The other cases are handled similarly. We start by reaching an all-negative distance profile as follows.
\begin{equation} \label{Eq1}
\begin{split}
(p,q,p+q)&\xrightarrow{(p+q,p,q)} (-q,p+q,p)\xrightarrow{(p,q,p+q)}(-(p+q),p,-q)\xrightarrow{(q,p+q,p)}(-p,-q,-(p+q)).
\end{split}
\end{equation}

As before, a value $x>0$ in the $i$-th coordinate represents a vertex in $B_{\alpha_i}$ at distance $x$ from $\xi$, whereas $-x$ represents a vertex in some $B_{\delta}$ with $\delta\neq\alpha_i$ at distance $x$ from $\xi$. A triplet $(r, s,t)$ above an arrow indicates that we move $r$ away from the  vertex in the first coordinate, $s$ from the one in the second coordinate and $t$ in the third.
Even permutations $\pi$ can be realized through an additional step
\[
 (-p,-q,-(p+q))\xrightarrow{(p+q,p,q)}(-q,-(p+q),-p)\ ; \  (-p,-q,-(p+q))\xrightarrow{(q,p+q,p)}(-(p+q),-p,-q).
\]
Our solution for odd $\pi$ depends on the fact that $p$ is even, viz.,
\begin{equation} \label{Eq2}
    (-p,-q,-(p+q))\xrightarrow{(p+q,q,p)}(-q,-p,-(p+q))\ ; \ 
(-p,-q,-(p+q))\xrightarrow{(p+q,p,q)}(-(p+q),-q,-p).
\end{equation}
Finally, to reach $(-p,-(p+q),-r)$, we proceed as follows:
\[
(-q,-p,-(p+q))\xrightarrow{(p+q,q,p)}(-p,-(p+q),-q).
\]
Thus, $5$ steps suffice to move from any vertex in $B_{\alpha_1}\times B_{\alpha_2}\times B_{\alpha_3}$ to any vertex in $B_{\delta_1}\times B_{\delta_2}\times B_{\delta_3}$ where $\delta_i\neq\alpha_i$ for all $i$. Consequently $\text{diam}(L)\le 10$ when $p$ is even and $r=p+q$. Similar bounds apply for all relevant choices of $p,q,r$.
\end{proof}

\begin{examples}
{\rm
Again we use Figure \ref{Tree_With_Marks} to illustrate the arguments in the preceding discussion. Pick $d=3,p=q=2,r=p+q=4$, and let us find a path from $(x_{11},x_{11},x_{1111})$  to $(x_{31},x_{2212},x_{11})$. We do not necessarily find a shortest path, since we insist on using a pivot vertex. In order to satisfy the requirement that $\delta_i\neq\alpha_i, \beta_i$ the pivot must belong to $B_2\times B_3 \times B_j$ for some $j\neq 1$. Let us choose, e.g., $(x_{21},x_{31},x_{2111})$ as pivot. By the recipe from the previous proof, we can proceed as follows:
\[
\begin{split}
(x_{11},x_{11},x_{1111})&\xrightarrow{(4,2,2)} (x_{21},x_{1111},x_{11})\xrightarrow{(2,2,4)}(x_{2111},x_{11},x_{21})\xrightarrow{(2,4,2)}(x_{21},x_{31},x_{2111})\\
(x_{31},x_{2212},x_{11})&\xrightarrow{(4,2,2)} (x_{22},x_{22},x_{1111})\xrightarrow{(2,4,2)}(x_{2211},x_{31},x_{11})\xrightarrow{(2,2,4)}(x_{22},x_{3111},x_{22})\\
&\xrightarrow{(2,2,4)}{(x_{21},x_{31},x_{2111})}.
\end{split}
\]
Namely, we found a path of length $7$ between them. In the proof we  bound the distance from above by twice the longest path to the pivot.
}
\end{examples}

\begin{proof}[Proof of \Tref{Poly_are_exp}]
As \Pref{lambdaGS} shows, $G_S$ is a (global) expander. The analysis of its local spectral expansion is slightly more complex: 
Let $L$ be the link of $G_S$ (All the links are isomorphic as graphs). Denote its edge expansion by $\delta$ and its spectral expansion by $\epsilon$.
Cheeger's inequality
\[
\delta\leq \sqrt{b_S^2-\lambda (L)^2}
\]
yields 
\[
\lambda(L)\leq \sqrt{b_S^2-\delta^2}
\]
and therefore $\epsilon\geq 1-\sqrt{1-\left(\frac{\delta}{b_S}\right)^2}$. For fixed $S=[p,q,p+q]$ with $p$ and $q$ even, $b_S$ depends only on $d$. Next we derive two lower bounds on $\epsilon$, called $\beta_1(d), \beta_2(d)$. Let $\gamma(d)=\max(\beta_1(d), \beta_2(d))$. Clearly, $\epsilon\ge \gamma(d)$. We obtain a lower bound on the spectral expansion that does not depend on $d$, namely the minimum of $\gamma$. While the qualitative statement of the theorem already follows by considering only $\beta_2$, involving $\beta_1$ improves the actual bound.

Babai \cite{Babai} attributes to Aldous the observation that every vertex transitive graph with diameter $\Delta$ has vertex expansion at least $\frac{1}{2\Delta}$. As shown, our graph has diameter $\leq 10$, so that $\delta\geq \frac{1}{20}$, and hence $\epsilon\geq 1-\sqrt{1-\left(\frac{1}{20b_S}\right)^2}$. This bound $\beta_1=\beta_1(d)$ is a decreasing function of $d$.

To derive the second lower bound $\beta_2(d)$, we return to our calculation of the diameter. The argument involving the pivot vertex yielded the connectivity of the link. However, the same idea implies much more. Namely, every two (not necessarily distinct) vertices of the link $\bar{x}=(x_1,x_2,x_3)$ and $\bar{y}=(y_1,y_2,y_3)$ are connected by {\em many} paths of length $10$. Let 
\[
C_i=\bigcup_{B_j\cap\{x_i,y_i\}=\emptyset} B_j.
\]
These $C_i$'s are exactly all connected components of $\mathbb{T}_d\setminus\{\xi\}$ disjoint from the $i$-th coordinates of $\bar x, \bar y$. The set $\bar{C}=C_1\times C_2\times C_3$ contains all possible pivots from the previous argument. Clearly the number of paths from $\bar{x}$ to $\bar{C}$ is the same as the from $\bar{y}$ to $\bar{C}$. Moreover, the number of length $5$ paths from $\bar{x}$ to any vertex in $\bar{C}$ is the same, and if this number is $n$, then there are at least $m:=\frac{n^2}{|\bar{C}\cap L |} \geq\frac{n^2}{|L|}$ paths of length $10$ from $\bar{x}$ to $\bar{y}$. We now seek a lower bound on $n$. Each arrow in (\ref{Eq1}) and (\ref{Eq2}) represents at least $(d-2)^{p+q}$ possible edges.
In addition, when $p,q$ and $r$ are all even, one can use the move
\begin{equation}\label{Eq3}
     (p,q,r)\xrightarrow{(p,q,r)}(p,q,r),
\end{equation}
which also represents at least $(d-2)^{p+q}$ edges.
Therefore, we can arrive at an all-negative distance profile in \textbf{exactly} $5$ moves (If we arrive too early, we use steps (\ref{Eq3})). Thus the number of such paths is at least $(d-2)^{5p+5q}$. By symmetry, a fraction $(\frac{d-2}{d-1})^3$ of these paths fall in $\bar{C}$, so that $$n\geq (d-2)^{5p+5q}\left(\frac{d-2}{d-1}\right)^3.$$
Recall that the number of vertices in $L$ is $$a_S= 6d^3(d-1)^{2p+2q-3}\ \  \textrm{or}\ \ 3d^3(d-1)^{2p+2q-3},$$ and its regularity is $b_S\leq 36d^{p+q}$.
Therefore, the number of length $10$ paths from $\bar{x}$ to $\bar{y}$ is at least  
\[
m\geq\frac{(d-2)^{10p+10q+6}}{6d^{2p+2q+6}}.
\]
\\
Let $X$ and $Y$ be $N\times N$ non-negative symmetric matrices with zero trace whose Perron eigenvector is the all-ones vector $\overrightarrow{1}$, and let $Z=X+Y$. Then $Z$'s second eigenvalue equals $\lambda_2(Z)=\max uZu^T$, where the maximum is over all unit vectors $u$ that are orthogonal to $\overrightarrow{1}$. Therefore,
$\lambda_2(Z)\leq \lambda_2(X))+\lambda_2(Y)$. A similar argument yields $0>\lambda_N(Z)\geq \lambda_N(X)+\lambda_N(Y)$. If, as usual we denote $\lambda:=\max\{\lambda_2, |\lambda_N|\}$, we conclude that $\lambda(Z)\leq \lambda(X)+\lambda(Y)$. Apply this to $X=(A^{10}-mJ), Y=mJ$ where $A$ is $L$'s adjacency matrix, to conclude that 
\[
\lambda(L)^{10}\leq \lambda(A^{10}-mJ)+\lambda(mJ)\leq \deg(A^{10}-mJ)=\deg(L)^{10}-m\overbrace{|L|}^{=a_S}\ ,
\]
since $\lambda(mJ)=\lambda(J)=0$ and the second largest eigenvalue cannot exceed the degree. Therefore
\[
\left(\frac{\lambda(L)}{\deg(L)}\right)^{10}\leq 1-\frac{ma_S}{b_S^{10}}\leq 1-\frac{(d-2)^{12p+12q+6}}{6^{20}d^{12p+12q+6}}
\]
and thus
\[
\epsilon\geq1-\frac{\lambda(L)}{\deg(L)}\geq 1-\sqrt[10]{1-\frac{(d-2)^{12p+12q+6}}{6^{20}d^{12p+12q+6}}},
\]
which is an increasing function of $d$. Together with the previous bound this proves our claim.
\end{proof}

\begin{rem}
This method applies to other choices of $S$, but we do not elaborate on that. Numerical calculations suggest that the actual normalized spectral gap does not even depend on $d$. If true, a proof of this would require another approach.
\end{rem}

\section{Some Concrete Explicit Constructions}\label{sec:concrete}

\subsection{The case $S=[1,1,0]$}  Recall the concrete description of $G_S$ from the introduction: Take three copies of a $d$-regular non-bipartite graph $G$ of girth bigger than $3$ and have a token move on each of them. At every step two of the tokens move to a neighboring vertex and the third token stays put. The resulting graph $G_{S}$ is $(3d^2,2d)$-regular, it is connected (Claim \ref{G_S-connected}) and has connected links (Claim \ref{G_S_v-connected}). Let $v_1,\ldots,v_d$ be the neighbors of  $v_0\in V(G)$.
We turn to study the spectrum of $L=(G_S)_{(v_0,v_0,v_0)}$, the link of $(v_0,v_0,v_0)$ in $G_S$. The graph $L$ is tripartite with parts $V_1, V_2, V_3$, where
$V_1= \{(v_0,v_i,v_j)\ |\ 1\le i,j\le d\}$, and likewise for $V_2, V_3.$ Edges between $V_2$ and $V_3$ are defined via $(v_i,v_j,v_0)\sim (v_{i'},v_0,v_{j'})$ iff $i'=i$. The two other adjacency conditions are similarly defined.(E.g. see Figure 3 for the $d=3$ case).

\begin{lem} \label{specL110}
The eigenvalues of $L$ are $2d,d,0,-d$ with multiplicities $1, 3(d-1), 3(d-1)^2, 3d-1$ respectively.
\end{lem}

\begin{proof}
Since eigenspaces corresponding to distinct eigenvalues are mutually orthogonal, it suffices to provide a total of $3d^2$ linearly independent eigenvectors with the appropriate eigenvalues.

The simple eigenvalue $2d$ corresponds to the all-$1$'s vector
Since $L$ is connected and $2d$-regular. 
	
The eigenspace of $d$: For $\alpha\in\{1,2,3\}$ and $\beta\in\{1,\ldots,d\}$, let $S_{\alpha, \beta}$ be the set of those vertices in $L$ whose $\alpha$-th coordinate is $v_{\beta}$. Note that the subgraph of $L$ spanned by $S_{\alpha, \beta}$ is a complete bipartite graph $K_{d,d}$. Also, $V_{\alpha}$ is an independent set, so for $\beta\neq\gamma$, there are no edges between $S_{\alpha, \beta}$ and $S_{\alpha, \gamma}$. It follows that for every $\beta\neq 1$, the vector $\chi_{A}-\chi_{B}$  is an eigenvector with eigenvalue $d$, where $A=S_{\alpha, 1}$ and $B=S_{\alpha, \beta}$. It is easily verified that these $3(d-1)$ vectors are  linearly independent.

Here is a list of $3(d-1)^2$ linearly independent vectors in the eigenspace of $0$. Each of these vectors has a support of $4$. For every $i,j\neq 1$ we take the vector with a $1$ coordinate at $(v_1,v_1,v_0), (v_i,v_j,v_0)$ and $-1$ at $(v_1,v_j,v_0),(v_i,v_1,v_0)$. The three cyclic permutations of these vectors are included as well.

Finally we give $3d$ vectors with eigenvalue $-d$. They sum to the zero vector and this is the only linear relation that they satisfy. By omitting one of them we have $3d-1$ linearly independent eigenvectors with eigenvalue $-d$. We describe $d$ of the vectors and get the factor of $3$ by rotations. The $i$-th vector in this list ($i=1,\ldots,d$) has the form $\chi_A-\chi_B$, where $A=\{(v_i,v_0,v_j)|j=1,\ldots,d\}$, and $B=\{(v_i,v_j,v_0)|j=1,\ldots,d\}$.
\end{proof}

\subsection{The case $S=[1,2,3]$} Here $G_{S}$ is $(6d^3(d-1)^3,2(d-1)^2(4d-7))$-regular. It is connected and so are its links. It also has interesting spectral properties, since by \Pref{lambdaGS}, if $G$ is Ramanujan, then $\lambda(G_S)\leq 12 d^{2}(d-1)^{7/2}$.
Actually, a similar conclusion can be drawn whenever $G$ has a substantial spectral gap.

\section{The perspective of high dimensional expansion}\label{sec:hdx}

The study of $(a,b)$-regular graphs can be cast in the language of simplicial complexes. Let us recall some basic facts from that theory. Let $X$ be a simplicial complex, and $\sigma$ a face in $X$. The {\em link} of $\sigma$ in $X$ is the following simplicial complex:
$$X_\sigma=\{\tau \in X\ |\ \sigma\cap\tau=\emptyset,\ \sigma\cup\tau\in X\}.$$
The $i$-th {\em skeleton} $X^{(i)}$ of $X$ is the simplicial complex that is comprised of all faces of $X$ of dimension $\le i$.\\
Associated with a graph $G=(V,E)$ is its {\em clique complex} $\ \mathcal{C}_G$, whose vertex set is $V$ and $S\subseteq V$ is a face of it if and only if $S$ spans a clique in $G$. Hence $G$ is $(a,b)$-regular if and only if the $1$-skeleton $\mathcal{C}_G^{(1)}$ is $a$-regular and the link of every vertex $v\in \mathcal{C}_G^{(2)}$ is a $b$-regular graph. If $G$ is an $(a,b)$-regular graph, then the number of $2$-faces in $\mathcal{C}_G^{(2)}$ is $\frac{abn}{6}$.

This section contains both negative and positive results. The negative results are mainly about the $[1,1,0]$-polygraphs and the positive ones are about $[1,2,3]$-polygraphs.

There is a considerable body of research, mostly quite recent on expansion in high dimensional simplicial complexes. Several different ways were proposed to quantify this notion. For the definitions of cosystolic and coboundary expansion, see e.g., \cite{EvraKaufman}.

\subsection{$[1,1,0]$-polygraphs have poor discrepancy}
For every base graph $G$, there are two sets in $G_{[1,1,0]}$, each containing $\frac{1}{8}$ of the vertices with no edges between them. Namely, let $A\subseteq V(G)$ of size $|A|=\frac{1}{2}|V(G)|$. Clearly, there are no edges between $A^3=\{(a_1,a_2,a_3) \colon a_1,a_2,a_3\in A\}$ and $(A^c)^3$. A similar construction can be given whenever $S$ contains a zero. In contrast, $[1,2,3]$-polygraphs exhibit better discrepancy properties, and in particular have the geometric overlap property, see below.
 
\subsection{Coboundary expansion}
This part is inspired by work in progress of Luria, Gundert and Rosenthal (e.g., Section $3$ of the lecture notes \cite{ConNotes}). They showed that Conlon's hypergraph \cite{Conlon} contains small non-trivial cocycles and thus is not a cosystolic expander and a fortiori not a coboundary expander either. Here we only provide a cocycle which is not a coboundary.

Again, let $G$ be a non-bipartite $d$-regular graph with girth larger then $3$, and $\Gamma=G_{[1,1,0]}$. We exhibit a set $A\subseteq E(\Gamma)$ such that:
\begin{enumerate}
\item Every triangle in $\Gamma $ has exactly two edges from $A$;
\item $A$ is not a cut in $\Gamma$.
\end{enumerate}
It follows that the characteristic function of $A$ is a non-trivial cocycle, implying that the  $2$-skeleton of $\Gamma$'s clique complex has a non-trivial first $\mathbb{F}_2$-cohomology and is thus not a coboundary expander.

The distance profile of every edge in $\Gamma$ is one of three: $(0,1,1)$, $(1,1,0)$ or $(1,0,1)$, and every triangle has exactly one edge of each kind. The set $A$ of those edges whose profile is $(1,1,0)$ or $(1,0,1)$ clearly satisfies condition $(1)$. To show condition $(2)$ we find an odd cycle in the graph $(V(\Gamma),A)$. Since $G$ is non-bipartite, it has an odd cycle, say $v_1,...,v_{\ell},v_1$. But then $$(v_1,v_1,v_1),(v_2,v_2,v_1),...,(v_{\ell},v_{\ell},v_1),(v_1,v_1,v_1)$$ is an odd cycle in $(V(\Gamma),A)$.\\
We note that this argument fails for $S=[1,2,3]$. On the other hand, this argument does work for $S=[1,1,2]$, showing that even a zero-free $S$ need not yield coboundary expansion.

\subsection{Geometric overlap property}
Let $X$ be a $2$-dimensional simplicial complex. Consider an embedding of $V(X)\to \mathbb{R}^2$ with the induced affine extension to $X$'s edges and faces. If for every such an embedding there is a point in $\mathbb{R}^2$ that meets at least an $\alpha$-fraction of the images of $X$'s $2$-faces, we say that $X$ has the $\alpha$-{\em geometric overlap} property. Work in this section and the following one is inspired by \cite{Conlon}. Here is our main theorem on this subject.

\begin{thm} \label{GeoOver}
Let $G$ be a $d$-regular graph with girth larger then $9$, $d>d_0$ and $\lambda(G) < \epsilon_0 d$. Then $G_{[1,2,3]}$ has the $\alpha_0$-geometric overlap property. Here $d_0, \epsilon_0, \alpha_0 >0$ are absolute constants.
\end{thm}

\begin{proof} The first ingredient of our argument comes from Bukh's proof of the Boros-F\"uredi theorem \cite{Bukh}. A {\em fan} of three lines in the plane that pass through a point $x$ splits $\mathbb{R}^2$ into $6$ sectors. For every finite $X \subset \mathbb{R}^2$, there is such a fan where each sector contains at least $\lfloor\frac{|X|}{6}\rfloor$ points of $X$. But then $x$ resides in every triangle whose three vertices come from non-contiguous sectors of the fan. 
Thus it suffices to show that if $A,B, C\subset V(G_{[1,2,3]})$ are disjoint subsets of size $\lfloor\frac{|V(G_{[1,2,3]})|}{6}\rfloor=\lfloor\frac{n^3}{6}\rfloor$ each, then a constant fraction of the triangles in $G_{[1,2,3]}$ are in $T(A,B,C)$, i.e., they meet $A, B$ and $C$. 

Using the expander mixing lemma (=EML), we derive an estimate of $|E_{[1,2,3]}(A,B)|$ and show that the density of the $A, B$ edges is very close to the overall edge density of $G_{[1,2,3]}$. We assign a {\em midpoint} to every {\em directed} $A\to B$ edge $u\to v$ (recall subsection \ref{Triangles}). The crucial property of this midpoint is that $d^3-O(d^2)$ of its $d^3$ neighbors {\em in $G_{[1,1,1]}$} form together with $u, v$ a triangle {\em in $G_{[1,2,3]}$}. This gives us a good lower bound on the number of triangles in $G_{[1,2,3]}$ that have exactly one vertex in $A$ and one in $B$. Next we need to show that however we choose $C$, many of these triangles have a vertex also in $C$. To this end, we apply the EML to $M$ and $C$ in $G_{[1,1,1]}$, where $M$ is the multiset of all midpoints created as above. Here $C$ is an arbitrary set of $\lfloor\frac{n^3}{6}\rfloor$ vertices outside $A\cup B$.

\noindent
We turn to carry out this plan now. By \Pref{lambdaGS}, $\lambda(G_{[1,2,3]})\leq 6 \mu d^{2}(d-1)^3$,  where $\mu =\max(\lambda(G),2\sqrt{d-1})$, and the EML yields:
\[
\frac{n^3}{6}(d-1)^3d^2(d+6 \mu)\ge
|E_{[1,2,3]}(A,B)| \geq \frac{n^3}{6}(d-1)^3d^2(d-6 \mu)
\]

Let $u=(u_1,u_2,u_3)\in A$, $v=(v_1,v_2,v_3)\in B$ be neighbors in $G_{[1,2,3]}$ and suppose that their distance profile is $(1,2,3)$, in this order. The vertices $u_2, v_2$ have a unique common neighbor in $G$, called $w$. Also, let $z_1, z_2$ be the vertices on the shortest path from $u_3$ to $v_3$. Then $\frak{m}=(u_1,v_2,z_{1})$ is a midpoint of the directed edge $u\to v$. Let $x=(x_1,x_2,x_3)$ be a neighbor of $\frak{m}$ in $G_{[1,1,1]}$, i.e., $x_1  u_1, x_2 v_2, x_3 z_{1}\in E(G)$. It is easily verified that if in addition $x_1\neq v_1$, $x_2\neq w$, and $x_3\neq u_3,z_{2}$, then $uvx$ is a triangle in $G_{[1,2,3]}$. Clearly $(d-1)^2(d-2)=d^3-O(d^2)$ of the $d^3$ neighbors of $\frak{m}$ satisfy these additional conditions. Figure \ref{fig:TriG123} provides a local view of the three factors of $G_{[1,2,3]}$.
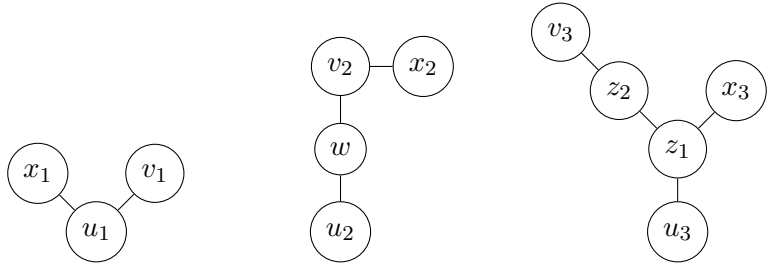
\begin{figure}[]
    \centering

    \caption{A triangle in $G_{[1,2,3]}$ as viewed in $G^{\square 3}$.}
    \label{fig:TriG123}
\begin{center}
\begin{tikzpicture}[scale=1.1]
\node[draw,shape=circle] (v0) at (0:0) {$u_1$};
\node[draw,shape=circle] (v1) at (45:1) {$v_1$};
\node[draw,shape=circle] (v2) at (135:1) {$x_1$};
\draw (v0) to (v2) ;
\draw (v0) to (v1) ;
\end{tikzpicture}
\qquad \qquad 
\begin{tikzpicture}[scale=1.1]
\node[draw,shape=circle] (w0) at (0:0) {$v_2$};
\node[draw,shape=circle] (w1) at (0:1) {$x_2$};
\node[draw,shape=circle] (w2) at (270:1) {$w$};
\node[draw,shape=circle] (w3) at (270:2) {$u_2$};
\draw (w0) -- (w1)
(w0) -- (w2)
(w2) -- (w3);
\end{tikzpicture}
\qquad  \ 
\begin{tikzpicture}[scale=1.1]
\node[draw,shape=circle] (w31) at (0:0) {$z_{1}$};
\node[draw,shape=circle] (u3) at (270:1) {$u_3$};
\node[draw,shape=circle] (x3) at (45:1) {$x_3$};
\node[draw,shape=circle] (w32) at (135:1) {$z_{2}$};
\node[draw,shape=circle] (v3) at (135:2) {$v_3$};
\draw (w31) -- (u3)
(w31) -- (x3)
(w31) -- (w32)
(w32) -- (v3);
\end{tikzpicture}
\end{center}
\end{figure}

Clearly the midpoint we chose for $v\to u$ differs from the one we choose for $u\to v$. Let $M=M(A,B)$ be the multiset of all such midpoints (for both $A\to B$ and $B\to A$ edges). Let $C$ be a set of $\lfloor\frac{n^3}{6}\rfloor$ vertices outside $A\cup B$. To count $M,C$ edges we need a version of the EML that applies as well to multisets of vertices.
\begin{lem}[\cite{Conlon}]\label{expmixlem}
Let $P, Q$ be two multisets of vertices in a $D$-regular $N$-vertex graph $H$. Then:
\[
\left|E(P,Q)-\frac{D}{N}|P||Q|\right|\le\lambda(H)\sqrt{\left(\sum_{x\in P}w_{x}^{2}-\frac{|P|^{2}}{N}\right)\left(\sum_{y\in Q}w_{y}^{2}-\frac{|Q|^{2}}{N}\right)}
\]
where $w_{x}, w_{y}$ is the multiplicity of $x\in P$ resp.\ of $y\in Q$.
\end{lem}
By \Pref{lambdaGS}, $\lambda(G_{[1,1,1]})\leq \mu d^{2}$, hence
\[
|E_{[1,1,1]}(M,C)|
\geq \frac{d^3}{n^3}|M||C|-\mu d^{2}\sqrt{|C|\sum_{y\in M}w_{y}^{2}}.
\]
As noted before, if $uv$ is an edge in $G_{[1,2,3]}$ and $\frak{m}$ is the midpoint we chose, then out of the $d^3$ neighbors that $\frak{m}$ has in $G_{[1,1,1]}$, at least $(d-1)^2(d-2)$ form a $G_{[1,2,3]}$-triangle with $u, v$. Therefore,
\[
|T(A,B,C)|\geq |E_{[1,1,1]}(M,C)|-(d^3-(d-1)^2(d-2))|M|
\]
A long but routine calculation shows that the theorem holds, e.g., with $d_0=1600$, $\epsilon_0=1/20$ and $\alpha_0=1/100$.
\end{proof}

\begin{rem}
\rm
So far we have provided no systematic explanation for the connection between the multiset $S=[1,2,3]$ and $[1,1,1]$. We turn to discuss this issue. Our argument utilizes two properties of $S$:
\begin{enumerate}
    \item All its elements are positive;
    \item Given a list of all the potential length profiles of triangles in $G_S$, one should check whether there is a triangles that has a center with a distance profile from one of the triangle's vertices which is zero-free. 
\end{enumerate} 
These are the only conditions we used about $G_S$. Thus, $[1,1,0]$ fails condition $(1)$, and $[1,1,2]$ does not satisfy $(2)$, but $[1,2,3]$ has such triangles, namely those with our specific choice of midpoints as centers. There are many other examples such as $[2,2,2]$. A nice aspect of the latter example is that for large enough $d$ it has the geometric overlap property even though its link is not even connected.
\end{rem}

\subsection{Mixing of the edge-triangle-edge random walk}
Random walks play a key role in the study of expander graphs, and similar questions are being studied in the high-dimensional realm as well, e.g., \cite{KM}, \cite{LLP}, \cite{DinurKaufman} and \cite{Conlon}. Consider the following random walk on the one-dimensional faces (i.e., edges) of a $2$-dimensional simplicial complex $X$. We move from an edge $e\in X^{(1)}$ to an edge that is chosen uniformly among all edges $f\in X^{(1)}$ with $f\cup e\in X^{(2)}$ (a triangle in $X$). The main issue here is to decide for simplicial complexes of interest whether this walk mixes rapidly. 

Differently stated, this is a walk on $\textrm{Aux}(X)$, a graph with vertex set $X^{(1)}$, where $ef$ is an edge iff $e\cup f\in X^{(2)}$. Given a multiset $S$, we consider $\textrm{Aux}(X)$ for $X=\mathcal{C}_{G_{S}}$, the clique complex of $G_S$. We establish a spectral gap for this graph (and hence rapid mixing of the walk) for $S=[1,1,0]$. A similar, but slightly harder argument applies as well to $S=[1,2,3]$: The duality in this case is between paths of length $6$ and paths of length $2$ and $4$, but the argument works the same.

\begin{lem} \label{edgeexpaux}
Let $G$ be a $d$-regular triangle-free graph with $n$ vertices, where $\lambda(G)=(1-\epsilon)d$, and let $\Gamma=\mathcal{C}_{G_{[1,1,0]}}^{(2)}$. Then $\lambda_2(\textrm{Aux}(\Gamma))\leq (1-\Omega(\epsilon^4))4d$. 
\end{lem}

\begin{proof}
Let $F\subseteq E_{[1,1,0]}$ be a set of at most a half of $G_{[1,1,0]}$'s edges and consider the triangles of $G_{[1,1,0]}$ that have edges from $F$. Proving a lower bound on $\textrm{Aux}(\Gamma)$'s edge expansion entails showing that for every such $F$, a constant fraction of these triangles are not contained in $F$.

Let $\Gamma$ denote the set of triangles in $\mathcal{T}$. We freely refer to them either as triangles or as $2$-faces of a complex. Each triangle is associated to its center, so that $\mathcal{T}=\bigsqcup_{x\in \Gamma^{(0)}}\mathcal{T}_x$, where $\mathcal{T}_x$ is the set of those triangles in $\Gamma$ whose center is the vertex $x$ (see Section \ref{Triangles}). We think of $\mathcal{T}_x$ as a simplicial complex with $d^3$ triangles, $3d^2$ edges and $3d$ vertices.

Every edge $xy$ in $G_{[1,1,0]}$ has two midpoints, call them $x'$ and $y'$. It is easy to verify that $x'$ and $y'$ are neighbors in $G_{[1,1,0]}$ and the midpoints of the edge $x'y'$ are $x$ and $y$. This yields a natural {\em duality} $xy\longleftrightarrow x'y'$ among the edges of $G_{[1,1,0]}$. Let $F_x$ be the set of those edges in $F$ that belong to a triangle in $\mathcal{T}_x$. Note that $2|F|=\sum_{x\in \Gamma^{(0)}}|F_x|$, since every edge of $G_{[1,1,0]}$ has exactly two midpoints, each of which is a center of $3d^2$ triangles.

Let $X$ be the set of those vertices $x$ of $\Gamma$ such that the vast majority of edges in $\mathcal{T}_x$ belong to $F$, i.e., $X:=\{x\in \Gamma^{(0)} \ |\ |F_x|\geq (1-\delta)3d^2\}$. We intend to show that with a proper choice of $\delta >0$ there holds 
\begin{equation} \label{ClClCl}
    \sum _{x\not\in X}|F_x| \geq \frac{\epsilon}{12}|F|.
\end{equation}

We can assume that $\sum_{x\in X} |F_x|\geq |F|$, for otherwise $\sum _{x\not\in X}|F_x|\ge|F|$ and (\ref{ClClCl}) clearly holds. Consequently, $|X|\geq \frac{|F|}{3d^2}$, since $|F_x|$ is smaller then the number of edges in $\mathcal{T}_x$ which is $3d^2$ for every $x$. We take $\delta<\frac{1}{4}$, so that: 
\[
\frac{3}{4}|X|\cdot 3d^2\leq|X|≥(1-\delta)3d^2\leq \sum_{x\in X}|F_x|\leq 2|F|\leq  \frac{3d^2n^3}{2}.
\]
Hence $|X|\leq \frac{2n^3}{3}$ and $|X^c|\geq \frac{1}{2}|X|$.
It is well known that a $k$-regular graph whose second eigenvalue is $\mu$ has edge expansion $\ge\frac{k-\mu}{2}$. Since $\lambda (G_{[1,1,0]})= d^2+2d\lambda(G)=3d^2-2\epsilon d^2$, we conclude that 
\[
|E(X,X^c)|\ge\epsilon d^2\min(|X|,|X^c|)\geq \frac{\epsilon d^2}{2}|X| .
\]
To derive an upper bound on $|E(X,X^c)|$, let $xy\in E(X,X^c)$, with $x\in X, y\not\in X$ and let $e$ be the edge that is dual to $xy$.
Clearly $e$ is in both $\mathcal{T}_x$ and $\mathcal{T}_y$, so either  $e\in F$ and $e\in F_y$ or $e\not\in F$ and $e\in\mathcal{T}_x\setminus F_x$. Therefore
\[
\sum_{y\in X^c}|F_y|+\sum_{x\in X}|\mathcal{T}_x\setminus F_x|\ge |E(X,X^c)|.
\]
But if $x\in X$, then $|\mathcal{T}_x\setminus F_x|\leq 3\delta d^2 $. We sum the above inequalities and conclude that
\[
\sum_{y\in X^c}|F_y|\geq\frac{\epsilon d^2}{2}|X|-3\delta d^2|X|\ge (\frac{\epsilon}{6}-\delta)|F|.
\]
By choosing $\delta=\frac{\epsilon}{12}$ we obtain (\ref{ClClCl}).

We proceed to prove the main statement. The fact that $\textrm{Aux}(\mathcal{T}_x)$ is isomorphic to $((\mathbb{T}_d)_S)_{(\xi,\xi,\xi)}$ together with the spectral information in \Lref{specL110}, imply that $\textrm{Aux}(\mathcal{T}_x)$ has edge expansion $\ge\frac{d}{2}$. Hence if $x\in X^c$, then 
\[
|E_{\textrm{Aux}(\Gamma)}(F_x,\mathcal{T}_x\setminus F_x)|\geq \frac{d}{2}\min(|F_x|,|\mathcal{T}_x\setminus F_x|)\geq \frac{\epsilon d}{24}|F_x|.
\]
Therefore,  
\[
\begin{split}
|E_{\textrm{Aux}(\Gamma)}(F,F^c)|&=\sum_{x\in \Gamma^{(0)}}|E_{\textrm{Aux}(\Gamma)}(F_x,\mathcal{T}_x\setminus F_x)|
\geq \sum_{x\in X^c}|E_{\textrm{Aux}(\Gamma)}(F_x,\mathcal{T}_x\setminus F_x)\\
&\ge \sum_{x\in X^c}\frac{\epsilon d}{24}|F_x|
\geq \frac{\epsilon^2}{288}d|F|,
\end{split}
\]
where the last step uses Inequality \ref{ClClCl}. In other words, $\textrm{Aux}(\Gamma)$ has edge expansion $\ge \frac{\epsilon^2}{288}d$. 
But the second eigenvalue of a $k$-regular graph with edge-expansion $h$ is at most $\sqrt{k^2-h^2}$ (see  Appedndix B in \cite{RS07}). Since $\textrm{Aux}(\Gamma)$ is $4d$-regular, this yields
\[
\lambda_2(\textrm{Aux}(\Gamma))\leq (1-\frac{\epsilon^4}{3\cdot 10^6})4d.
\]

\end{proof}

\noindent{In order to control the low end of $\textrm{Aux}(\Gamma)$'s spectrum we recall the following}:

\begin{lem} [\cite{DR}] \label{DRDR}
Let $G=(V,E)$ be an $N$-vertex $D$-regular graph with eigenvalues $\lambda_1\ge\ldots\ge\lambda_N$. For $U\subseteq V$ let $b(U)$ denote the least number of edges that must be removed to make subgraph induced by $U$ bipartite. Then
\[
\lambda_{N}\ge-D+\frac{\Psi^{2}}{4D},
\]
where $$\Psi=\min_{U\ne\emptyset}\frac{b(U)+|E(U,U^{c})|}{|U|}.$$
\end{lem}

We can now establish a gap at the bottom of $\textrm{Aux}(\Gamma)$'s spectrum. We consider $U$ either as a set of vertices in $\textrm{Aux}(\Gamma)$, or a set of edges in $G_{[1,1,0]}$.
We separate the proof into two cases: 
\begin{itemize}
    \item When $U$ is very large, and therefore contains many triangles;
    \item When $U$ is not very large in which case we can apply \Clref{edgeexpaux}.
\end{itemize}

We need the following
\begin{claim}
A set $W$ of $w$ edges in $K_{d,d,d}$ contains at least $d(w-2d^2)^+$ triangles. The bound is tight.
\end{claim}
\begin{proof}
Tightness is easy. If $w\le 2d^2$, we can have $W$ completely avoid one of the three $K_{d,d}$'s, and therefore be triangle free. When $w> 2d^2$, have $W$ contain two of the $K_{d,d}$'s. Every edge in the third $K_{d,d}$ is in exactly $d$ triangles so $W$ has exactly $d(w-2d^2)$ triangles. The proof of the bound is very similar: Start with any set $W$ of $w> 2d^2$ edges and sequentially add to $W$ every remaining edge in $K_{d,d,d}$. The addition of a new edge creates at most $d$ new triangles, and eventually we reach the whole of $K_{d,d,d}$ with its $d^3$ triangles. If follows that we must have started with at least $d(w-2d^2)$ triangles, as claimed.
\end{proof}
We maintain the same notations: $U_x$ is the set of edges in $U$ that belong to a triangle in $\mathcal{T}_x$, the set of triangles with center $x$. A triangle is associated to its center, and we partition the triangles contained in $U$ according to their various centers. We also recall that the $1$-skeleton of $\mathcal{T}_x$ is a complete tripartite graph $K_{d,d,d}$.

Call vertex $x$ {\em heavy} if $|U_x|\ge\frac{5}{6}|E(K_{d,d,d})|= \frac{5d^2}{2}$, and note that by the above claim, in this case  $U_x$ must contain at least  $\frac{d^3}{2}$ triangles, which is also a lower bound on the number of triangles in $\textrm{Aux}(U_x)$. But all triangles in $\textrm{Aux}(U_x)$ are edge disjoint, so we must remove at least $\frac{d^3}{2}$ edges from $\textrm{Aux}(U_x)$ to make it bipartite.

Also recall that every edge in $E_{[1,1,0]}$ belongs to exactly two triangles. Consequently, if $|U|\geq \frac{71}{72}|E_{[1,1,0]}|$, then at lease $\frac 56$ of the vertices are heavy. Therefore, in this case we must remove at least $\frac{5n^3}{6}\cdot\frac{d^3}{2}=\frac{5d^3n^3}{12}$ edges to make the induced graph on $U$ bipartite. Therefore $$\frac{b(U)}{|U|}\geq \frac{5d^3n^3}{12}\cdot \frac{2}{3d^2n^3}=\frac{5d}{18}.$$

On the other hand, if $|U|\leq \frac{71}{72}|E_{[1,1,0]}|$, then $$|E_{\textrm{Aux}(\Gamma)}(U,U^c)|\geq \frac{d\epsilon^2}{192}\min(|U|,|U^c|)\geq  \frac{d\epsilon^2}{71\cdot192}|U| $$
and therefore $\frac{|E_\textrm{Aux}(U,U^c)|}{|U|}\geq \frac{d\epsilon^2}{2\cdot10^4}$.
We conclude that $\Psi \geq \frac{d\epsilon^2}{2\cdot10^4}$, and  by \Lref{DRDR}, $\lambda_N\geq -4d+\frac{\epsilon^4d}{32\cdot 10^{8}}$.
Since we established an additive gap of size $O(d\epsilon^4)$ both from above and from below for $\textrm{Aux}(\Gamma)$, it follows that the edge-triangle-edge random walk mixes rapidly.

\section{For which $(a,b)$ do large $(a,b)$-regular graphs exist?}\label{sec:geom_const}
This section provides a partial answer to the question in the title of this section. 
If $b=2$ and the links are connected, then every link is a cycle. So, the graph in question is the $1$-skeleton of a triangulated $2$-manifold. A good example with which to start is $a=6$. The Cayley graph of $\mathbb{Z}^2$ with generators $(\pm1,0),(0,\pm1),\pm(1,1)$ is the planar triangular grid. The quotient of this graph $\bmod ~m\mathbb{Z}\times n\mathbb{Z}$, is a $(6,2)$-regular finite triangulation of the torus whose links are connected. We ask for which values of $a$ there exist infinitely many such graphs.

This discussion is closely related to the study of equivelar polyhedral $2$-manifolds and non-singular $\{p,q\}$-patterns on surfaces, a subject on which there exists a considerable body of literature. We only mention \cite{MSW1} and \cite{MSW2} where infinitely many such graphs for $a\geq 7$ are constructed. Some of these constructions are inductive and start from the above triangulations of the torus. Other constructions are iterative and use snub polyhedra of prisms. 

We recall that the \emph{tensor product} $G\otimes H$ of two graphs $G$ and $H$, is a graph with vertex set $V(G)\times V(H)$ where $(u,v)$ and $(u',v')$ are neighbors when both $uu'\in E(G)$ and $vv'\in E(H)$. Therefore its adjacency matrix is the Kronecker tensor product $A_G\otimes A_H$. Note that $G\otimes G$ is isomorphic to $G_{[1,1]}$. Also, if $G$ is $(a,b)$-regular and $H$ is $(a',b')$-regular, then $G\otimes H$ is $(aa',bb')$-regular.

Thus, if $G$ is $(k,2)$-regular, then $G\otimes K_m$ is $(k(m-1),2(m-2))$-regular. This yields arbitrarily large $(k(m-1),2(m-2))$-regular graphs with connected links. This means that the question in the title is answered positively for many $a>b$ and all asymptotic relations between $a$ and $b$.  
\section{Open Questions and Remarks}
Countless questions suggest themselves in this new domain of research. We mention below a few which we view as the most attractive.
\begin{enumerate}
\item {\bf A randomized model:} One of the earliest discoveries in the study of expander graphs is that in essentially every reasonable model of random graphs, and in particular for random $d$-regular graphs, almost all graphs are expanders. It would be very interesting to find a randomized model of $(a,b)$-regular graphs and in particular one where most members are expanders both locally and globally.
\item {\bf Higher-dimensional constructions:} We have touched upon the connections of our subject with the study of expansion in higher-dimensional simplicial complexes. Clearly, $(a,b)$-regularity is a two-dimensional condition, and we know essentially nothing for higher dimensions. Concretely: do $(a,b,c)$-regular graphs exist? Namely for fixed $a>b>c>1$, we ask whether there exist arbitrarily large $a$-regular graphs where the link of every vertex is $b$-regular, and the link of every edge is $c$-regular. We want, moreover, that the whole graph, every vertex link and every edge link be expanders. We stress that no such constructions  based on Ramanujan complexes \cite{LSV} are presently known. There are indications that the situation in dimension two is less rigid than in higher dimensions. Does this translate to some non-existence theorems?

\item {\bf Garland's method} \cite{Garland} is a powerful tool in the study of high-dimensional expansion, e.g., \cite{Openheim}, \cite{Papikian}). In order to apply the method for an $(a,b)$-regular graph $G$, it needs to have the property that the spectrum of every vertex link is contained in $\{-b,b\}\cup [-\beta, \beta]$ for some $\beta < \frac{b}{2}$. In such case, Garland's method asserts that $G$ is also a global expander.
However, some substantial new ideas will be needed to construct such an $(a,b)$-regular graph using only combinatorial arguments. For instance, polygraphs cannot have this property. Indeed, compare what happens when we start from a $d$-regular graph $G$ that is a very good expander vs.\ a very bad one. While $G_S$ inherits $G$'s expansion quality, the links of the two graphs are identical.

\item {\bf Trade off:} The lower bound on $\lambda_2$ from \Tref{tradeoff} is an increasing function of $\delta$. However, we do not know how tight this bound is and whether the {\em best possible} lower bound on $\lambda_2$ increases with $\delta$. \Tref{ABTB} is tight, so the best bounds for $\delta=0$ and $\delta>0$ differ. But whether the same holds as $\delta>0$ increases, we do not know.
\end{enumerate}

\section*{Appendix: Regularity of the links}

Recall that $\Omega$ is its set of all the arrangements of the multiset $S=[l_1,...,l_m]$.
For a positive integer $i$, define $f_i \colon \mathbb{Z}_{\geq 0}^2 \to \mathbb{Z}_{\geq 0}$ as follows: 
\[
f_i (j,k)=
\begin{cases}
0,\ &\textrm{for}\   i+j+k\equiv 1\pmod 2 \\
0,\ &\textrm{for}\ \frac{i+j+k}{2}<\max \{i,j,k\}\\
(d-1)^{\frac{j+k-i}{2}},\ &\textrm{for}\ \frac{i+j+k}{2}=\max \{i,j,k\}\\
(d-2)(d-1)^{\frac{j+k-i}{2}-1},\ &\textrm{otherwise},
\end{cases}
\]
and for $i=0$,
\[
f_0(j,k)=
\begin{cases}
0, &j\ne k\\
\lfloor d(d-1)^{j-1} \rfloor, &j=k.

\end{cases}
\]
Thus, by \Clref{b_S=0}, we conclude that
\[
b_{S}=\sum_{\omega,\omega'\in{\Omega}}\prod_{j=1}^m f_{l_j}(\omega_j,\omega'_j)).
\]

\section*{Acknowledgements:} We thank Irit Dinur, whose questions prompted us to work on these problems. Our conversations with Irit took place during a special year on high dimensional combinatorics held in the Israel Institute of Advanced Studies. We thank the IIAS and the organizers of this beautiful year.

\end{document}